\documentclass[12pt, reqno]{amsart}
\usepackage[utf8]{inputenc}

\usepackage{mathtools}
\usepackage{amssymb}
\usepackage{tikz}

\usetikzlibrary{cd}
\usetikzlibrary{decorations.pathmorphing}

\usepackage{xcolor}
\usepackage{ragged2e}
\usepackage[style=alphabetic,backend=bibtex, maxbibnames=99]{biblatex}
\renewbibmacro{in:}{}
\usepackage{hyperref}

\usepackage[margin=3.5cm]{geometry}

\theoremstyle{plain}
\newtheorem{proposition}{Proposition}
\newtheorem{theorem}{Theorem}
\newtheorem{lemma}{Lemma}

\theoremstyle{remark}
\newtheorem{remark}{Remark}

\newtheorem{definition}{Definition}

\newcommand{\CC}{\mathbb{C}}
\newcommand{\RR}{\mathbb{R}}

\newcommand{\HH}{\mathcal{H}}

\newcommand{\Pm}{\mathcal{P}}

\newcommand{\Ps}{{\Pm_{\rm Sch}}}

\newcommand{\Ppl}{{\Pm_{\rm F}}}

\renewcommand{\aa}{\mathfrak{a}}

\renewcommand{\gg}{\mathfrak{g}}
\newcommand{\kk}{\mathfrak{k}}
\newcommand{\mm}{\mathfrak{m}}

\newcommand{\nn}{\mathfrak{n}}
\renewcommand{\ss}{\mathfrak{s}}

\newcommand{\uu}{\mathfrak{u}}

\DeclareMathOperator{\Ad}{Ad}
\DeclareMathOperator{\ad}{ad}

\title[Different Schr\"odinger models and
FH Transform]{Unitary equivalence of quantization of different Schr\"odinger models and
Fourier--Helgason  Transform}

\author{A.C.~Ferreira}
\address{Ana~Cristina~Ferreira\\
Centro de Matem\'atica, Universidade do Minho, Campus de Gualtar, 4710-057 Braga, Portugal.}
\email{anaferreira@math.uminho.pt}

\author{J.~Hilgert}
\address{Joachim~Hilgert\\Department of Mathematics, Paderborn University, Paderborn, Germany.}
\email{joachim.hilgert@upb.de}

\author{J.M.~Mourão}
\address{Jos\'e~M.~Mourão\\ Department of Mathematics and Center for Mathematical Analysis, Geometry and Dynamical Systems, Instituto Superior T\'ecnico, Av. Rovisco Pais 1, 1049-001, Lisbon, Portugal.}
\email{jmourao@tecnico.ulisboa.pt}

\author{J.P.~Nunes}
\address{Jo\~ao~P.~Nunes\\Department of Mathematics and Center for Mathematical Analysis, Geometry and Dynamical Systems, Instituto Superior T\'ecnico, Av. Rovisco Pais 1, 1049-001, Lisbon, Portugal.}
\email{jpnunes@math.tecnico.ulisboa.pt}

\date{\today}

\begin{document}

\begin{abstract}

The Fourier--Helgason (FH) transform for a noncompact symmetric space $G/K$ establishes the direct integral decomposition of the 
unitary representation of $G$ on $L^2(G/K)$ into irreducible principal series representations.

By applying techniques of geometric quantization to the symplectic manifold $T^*(G/K),$ Lisiecki in 1987 gave a geometric 
interpretation of the FH transform in the case where $G$ is complex. 
He defined for general $G$ a ``horizontal" polarization on $T^*(G/K)$ and showed that, for complex $G$,  the Blattner-Kostant-Sternberg (BKS) pairing between the
Schr\"odinger vertical polarization Hilbert space, $L^2(G/K)$, and the Hilbert space of horizontally polarized functions 
coincides with the FH transform. However, in the same paper, Lisiecki showed that for noncomplex Lie groups 
the BKS pairing is not equivalent to the FH transform and is nonunitary in general.

In the present paper, we resolve this discrepancy between the FH transform and geometric quantization in the case when $G$ is not complex. 

First, we show that the horizontal  polarization is the infinite-time limit of the push-forward of the  vertical
polarization with respect to the geodesic flow for a $G$--invariant Riemannian metric.
Then we lift the geodesic flow to an intertwining unitary parallel transport on the quantum bundle that we call the quantum geodesic transform (QGT). Finally, we show that
the QGT has a well-defined limit, as the  geodesic time goes to infinity, and that it is equal to the FH transform, up to the phase of the Harish-Chandra $c$--function, a constant  phase given by an  eighth root of unity and a phase related to gauge fixing of the prequantum connection.

\end{abstract}

\maketitle

\newpage

\tableofcontents

\section{Introduction and main results}
\label{sect-introd}

Besides the standard model associated with the vertical, or Schr\"odinger, polarization, $\mathcal{P}_0$, defined by the projection 
\begin{equation}
\nonumber \pi : T^*\left( G/K \right) \longrightarrow G/K, \,  
\end{equation}
other
 Schr\"odinger models for the cotangent bundle of noncompact Riemannian symmetric spaces $T^*\left( G/K \right)$ have been
 the object of recent interest
 \cite{sakellaridis23, sakellaridis25}. 
 
 A real polarization is called a Schr\"odinger model if the quantum Hilbert space of polarized sections coincides, up to a phase, with the $L^2$--space of half-densities on the space of leaves of the polarization.
A very important Schr\"odinger  polarization, which generalizes the momentum polarization for $G=\mathbb{R}^n$, 
was introduced by Lisiecki in \cite{Lisiecki87}. Lisiecki called this polarization horizontal (we will call it the Fourier polarization, $\mathcal{P}_\mathrm{F}$, see Proposition \ref{prop:Fourier-polarisation}), and it is defined by the projection 
\begin{equation}
\nonumber  \widehat \pi : T^*(G/K)_{\mathrm{reg}} \cong (G/K) \times \mathfrak{a}^*_{+} \times B \longrightarrow \mathfrak{a}^*_{+} \times B, 
\end{equation}
where $B = K/M$. Lisiecki proved (see Theorem \ref{lisieckithm} below) that
 the Hilbert space of $\mathcal{P}_\mathrm{F}$-polarized quantum states is
$$
\mathcal{H}_{\mathcal{P}_F} \cong L^2\left (\aa_+^*\times B, \frac{d\lambda db}{\vert c(\lambda)\vert^2}\right) \, ,
$$   
where $c$ is the Harish-Chandra c-function.

The expectation that the two Schr\"odin\-ger quantizations are unitarily equivalent, was supported by Helgason's 1984 proof of the unitarity of the Fourier--Helgason (FH) transform~\cite{helgason:1984} (see Theorem \ref{th_FT}),
$$\mathcal{F} \, : \, L^2(X, dx)  \longrightarrow  L^2 \left(\mathfrak{a}_+^* \times  B, \frac{d\lambda db}{|c(\lambda)|^2}  \right)$$
{and was} rooted in the hope that the FH transform would, as in the case of the usual Fourier transform for $T^* \RR^n$, correspond to the unitary equivalence of the two quantizations.
A serious doubt on this expectation was however cast already by Lisiecki in \cite{Lisiecki87}, where he showed 
that the Blattner--Kostant--Sternberg (BKS) pairing map, between these two Schr\"odinger quantizations of $T^*(G/K)$,
 coincides with the FH transform only if $G$ is complex, i.e.\ $G \cong K_{\mathbb C}$. If $G$ is not complex, the BKS pairing map does not coincide with the FH transform, and is not unitary.

To address this issue, we first attempted a method that we had applied successfully in the case of compact symmetric spaces $T^*(U/K)$, namely considering
\begin{equation}
\Pm_{is} = (\varphi_{-is})_* \Ps,
\end{equation}
the push-forward of the vertical polarization by the imaginary-time geodesic flow. However, neither of the following two important properties, both of which hold for $T^*(U/K)$, is true in the case of $T^*(G/K)$:
\begin{itemize}
\item[(P1)] $\Pm_{is}$ is a K\"ahler polarization for every $s>0$;
\item[(P2)] $\displaystyle\lim_{s\to +\infty} \Pm_{is}$ exists and is the Fourier polarization $\Ppl$.
\end{itemize}
On regular points of $T^*(G/K)$ it can instead be shown that:
\begin{itemize}
\item[(P1')] $\Pm_{is}$ is a K\"ahler polarization only in a tubular neighborhood of the zero section, which leads to the difficult problem of quantizing crown domains;
\item[(P2')] The limit $\lim_{s\to \infty} \Pm_{is}$ does not exist.
\end{itemize}

It turns out, however, that if instead we take the geodesic flow in real time and let time go to infinity, we have
\begin{equation}
 \lim_{t\to\infty} \mathcal{P}_t = 
 \Ppl,
\end{equation}
where $\mathcal{P}_t = (\varphi_{-t})_* \Ps.$ In addition, a natural  description of the FH transform, corrected with the phase of $c$, via geometric quantization follows.

More precisely our main results are the following.\\

\noindent {\bf Theorem A} (Theorem \ref{thmpolarization}, in Section \ref{sectionevolutionofpol}) On the open dense subset $(T^*X)_\mathrm{reg}\subset T^*X$ (see Section \ref{subsechamiltonian}), the push-forward of the vertical polarization by the geodesic flow gives, in the infinite-time limit, the Fourier polarization $\mathcal{P}_\mathrm{F}.$
\bigskip

Similarly to  \cite{baier.ferreira.hilgert.mourao.nunes:2024} (Section 4.2)
we then lift the geodesic flow to a $G$-equivariant  parallel transport on the quantum bundle that we call quantum geodesic transform (QGT). 
An important advantage of the noncompact symmetric space case, $G/K$, studied in the present paper, over that of the compact symmetric space $U/K$, studied in \cite{baier.ferreira.hilgert.mourao.nunes:2024} is that, by considering the real time flow, the (finite time) parallel transport becomes automatically unitary.
That is, we obtain a $G$-equivariant unitary isomorphism between the Hilbert space for the (half-form corrected) vertically polarized functions 
and the Hilbert space of (half-form corrected) $\mathcal{P}_t$-polarized functions, which lifts the geodesic flow and intertwines the two quantizations
$$
U_t: \mathcal{H}_{\mathcal{P}_0} \cong L^2(G/K, dx)\stackrel{\sim}{\longrightarrow} \mathcal{H}_{\mathcal{P}_t},\, t\geq 0.
$$

Finally we show that
the QGT has a well-defined limit as the  geodesic time goes to infinity. 
The limit is given by a transform $\mathrm{U}$ (see Definition \ref{defU})  that is  unitarily equivalent  to the 
FH transform: the two transforms coincide up to a multiplicative phase given by the product of the phase of the Harish–Chandra \(c\)-function, by a phase function $e^{iS}$ which plays the role of a gauge transformation (see Lisiecki's Theorem \ref{lisieckithm} and (\ref{e33})) and by a constant phase given by an eighth root of unity. 

The phase factor $e^{iS}$ represents only the standard choice of trivialization of the prequantum line bundle on $T^*(G/K)$. 
The phase of the $c$--function however has, we believe, a deep meaning for the equivalence of the two quantizations. 
It is well known since the work of Harish-Chandra that the $c$-function encoding the asymptotic behavior of spherical functions plays a key role in Fourier inversion and Plancherel formulas for non-compact symmetric spaces. It was observed by Semenov-Tjan-\v Sanski\u i that this can be reinterpreted as the asymptotic behavior of solutions of the wave equation given as superpositions of plane waves and thus explains the central role of the $c$-function in the scattering theory of the Laplacian of symmetric spaces (see in particular \cite[Prop.~9]{Semenov-Tjan-Samski76}).  

Our second main result actually shows that, besides the phase of $c$, the Plancherel density appearing in the Fourier-Helgason transform and given in terms of the $c$-function also appears automatically in the limit of the quantum geodesic transform. It reads:

\bigskip

\noindent {\bf Theorem B} (Theorem \ref{ThQG}, in Section \ref{toinftyandbeyond}) Let ${U}_t^\epsilon$ be the regularized quantum geodesic transform (see Definition \ref{defregularizedU}) at time $t$. One has
$$
\lim_{\epsilon\to 0} \lim_{t\to +\infty} U^\epsilon_t = (-1)^{\vert\Sigma^+\vert/2} e^{-i\frac{\pi r}{4}}  \,\mathrm{U} \, : \, 
L^2(G/K, dx) \longrightarrow  L^2\left(\mathfrak{a}^*_+ \times K/M, \frac{d\lambda db}{|c(\lambda)|^2}\right),
$$
where $r$ is the rank of $G/K$ and $\vert \Sigma^+\vert$ is the number of positive roots counted with multiplicity. The operator $\mathrm{U}$ is essentially   given (see Definition \ref{defU}) by the Fourier-Helgason transform composed with the multiplication by the phase of $c$.

\bigskip

 The unitary equivalence of the two Schr\"odinger models, which we prove in the present paper,
 could be of interest to the theory of transfer operators in the Langlands program, when
    these operators are viewed as intertwining different Schr\"odinger models  for the 
    cotangent bundle of $ \mathfrak{X} = \left(X \times X \right)/G$, where $X = G/K$, realized as
    symplectic quotient of $T^*(X\times X)$ with respect to the diagonal action of $G$,
    $$
    T^* \mathfrak{X} 
=    \left(T^*X \times_{\mathfrak{g}^*} T^*X\right)/G \, ,
    $$
    (see \cite[\S 5]{sakellaridis23} and \cite[\S 1.1.1]{sakellaridis25}).
   It would also be of interest to relate the horospherical limits in \cite{sakellaridis25}
with the quantum geodesic transform at infinite geodesic time.

\section{Preliminaries}
\label{sect-prelim}

{
\subsection{Notation}
\label{Notation}
The following standard notation will be used throughout the paper.

$G$ denotes a (not necessarily complex) noncompact, connected, semisimple Lie group with finite center;

$e$ denotes the identity element of $G$;

$\gg$ denotes its Lie algebra;

$\sigma: G \longrightarrow G$ is a fixed Cartan involution; 

$\gg= \kk + \ss$ is a fixed Cartan decomposition of $\gg$ (with respect to $d_e \sigma: \gg \longrightarrow \gg$);

$\aa$ is a fixed maximal abelian subalgebra of $\ss$;

$\mm$ is the centralizer of $\aa$ in $\kk$; 

$\aa^\ast_+$ is a fixed open Weyl chamber;

$\Sigma$ is the set of restricted roots of $(\gg,\aa)$: for $\alpha \in \Sigma$, $\gg_\alpha$ is the corresponding root space; 

$\Sigma^+$ is the subset of positive roots corresponding to $\aa^\ast_+$;

$\nn = \sum_{\alpha\in \Sigma^+} \gg_\alpha $ and $\bar\nn = \sum_{-\alpha\in \Sigma^+} \gg_\alpha $;

$\gg= \kk + \aa + \nn$ is the Iwasawa decomposition of $\gg$;

$K$ is the analytic subgroup of $G$ with Lie algebra $\kk$ (a maximal compact subgroup of $G$);

$A= \exp \aa$;

$N= \exp \nn$;

$M$ is the centralizer of $A$ in $K$;

$MAN$ is the minimal parabolic subgroup of $G$ (its Lie algebra equals $\mm+\aa+\nn)$;

$X=G/K$ (Riemannian symmetric space of noncompact type), $x_o =e\cdot K$ (the ``origin'' of $X$);

$B=G/MAN = K/M$, $b_o=e\cdot MAN=e\cdot M$.

}

Throughout the paper, we will use  the Iwasawa decompositions of $G$,  $G=KAN$ and $G= NAK$ 
and we will write
\begin{equation}\label{iwasawas}
   G\ni g = k_1(g) e^{H(g)} n_1(g) = n_2(g) e^{A(g)} k_2(g),
\end{equation}
where, for $j=1,2$, $k_j(g)\in K$, $n_j(g)\in N$ and $H(g), A(g)\in \aa.$

Since $G$ is connected and semisimple, it is also unimodular and we fix a left- and right-invariant  Haar measure $dg$ on $G$. The induced left-$G$-invariant measure on $X$ is denoted by $dx$. More precisely, we let $dg$, and therefore also $dx$, be normalized as in \cite{helgason:2008}, Chapter II, Paragraph 3, Section 1. We will denote by $db$ the normalized $K$-invariant measure on $B$. The invariant  measures on $K$ and $M$ will be denoted, respectively, by $dk$ and $dm$, so that, from Theorem I.1.9 in \cite{helgason:1984}, one has
$$
\int_K {f(k) dk} = \int_B \left(\int_M f(km) dm \right) db.
$$
The Lebesgue measure on $\aa_+^*$ with normalization induced by the Killing form, and then multiplied by a factor of $(2\pi)^{-\frac12 \dim \aa}$, will be denoted by $d\lambda$, as in  \cite[\S~II.3.1]{helgason:2008}.

\subsection{Basic definitions}\label{subsec:basic definitions}

Let $G$ be a {non}compact connected semisimple Lie group with finite center
and $$\sigma:G\to G$$ an involutive automorphism 
such that $X=G/K$ is a symmetric space of {non-}compact type, where the group of fixed points $K=G^\sigma$ is a maximal compact subgroup of $G$. Note that $K$ is automatically connected, see \cite[\S~13.3]{HN12}. 
 
Note that the Killing form $\kappa$ on the Lie algebra $\gg$ of $G$ is an $\Ad_G$-invariant nondegenerate indefinite {pseudo-}inner product, which is positive definite on $\ss$. It defines a $G$-equivariant linear isomorphism $\gg$ and its real dual space $\gg^*$ which can be used to transfer $\kappa$ to $\gg^*$. We denote the resulting $\Ad^*_G$-invariant nondegenerate indefinite {pseudo-}inner product on $\gg^*$ also by $\kappa$. 
We obtain  orthogonal Cartan decompositions 
$$
\gg = \kk \oplus \ss, \,\, \gg^* = \kk^* \oplus \ss^*,
$$
where, for the derived automorphism of $\gg$, now also denoted by $\sigma$,
$$
{\sigma_{\vert{\kk}} = \mathrm{Id}_\kk}, \, \sigma_{\vert{\ss}} = -\mathrm{Id}_\ss,
$$
with $[\kk,\ss] \subset \ss$ and $[\ss,\ss] \subset \kk$. The linear isomorphism $\gg\cong \gg^*$ defined via $\kappa$ preserves this structure and in particular maps $\ss$ to $\ss^*$. Thus $\kappa$ is also positive definite on $\ss^*$.

The exponential map along $\ss$ followed by the canonical projection onto $X$,
$$\ss \stackrel{\exp}{\longrightarrow} G\to X,$$ is a diffeomorphism so that, in particular, $X$ is simply connected.

The cotangent bundle $T^\ast X$ has the structure of a {homogeneous} vector bundle associated to the principal $K$-bundle $G\to X$,
$$T^\ast X =G\times_K \ss^*,
$$
where 
\begin{equation}
\nonumber
[g,\xi] = [gk, \mathrm{Ad}^\ast_{k^{-1}}(\xi)],\,\, g\in G, \xi\in \ss^*, k\in K 
\end{equation}
and  $[g, \xi]$ denotes the $K$-orbit through the point $(g,\xi)$ in $G \times \ss^*$.

\subsection{$T^\ast  X$ as symplectic reduction of $T^\ast  G$}

Here, we will recall briefly standard facts about the symplectic geometry of the cotangent bundles $T^\ast  G$ and $T^\ast  X$, realized as the symplectic quotient
$$
  T^\ast  X = \left(T^\ast  G\right) /\!\!/ K \,.
$$

Recall that there is a Hamiltonian right action of $G$ on $T^\ast G$ with moment map $\mu^{(R)}$. The moment map of the corresponding right action of the subgroup $K\subset G$ is
\begin{equation}
\label{ee-muk}
 \mu^K = \pi_{\kk^\ast } \circ \mu^{(R)} ,
 \qquad \pi_{\kk^\ast }: \gg^\ast  \to \kk^\ast  .
\end{equation}
Using the {left-}trivialization $T^\ast G\cong G\times \gg^\ast  $ one finds that
\begin{eqnarray}
 \label{e-quot-const}
 \nonumber \left(\mu^K \right)^{-1}(0) &=& G \times \ss^\ast 
 \quad \text{ and } \\ 
 \left(T^\ast  G\right) /\!\!/ K &=&
 \left( \mu^K \right)^{-1}(0) / K =
 G \times_K \ss^\ast  
 \ =\ T^\ast X.
\end{eqnarray}

The moment map of the left $G$-action on $T^\ast  X$, which we will denote ${\mu}$, descends from $T^\ast G$ through this quotient,
\begin{equation}
\label{e-moment-symm-space}
 \begin{tikzcd}
  G \arrow[r, phantom, "{\circlearrowright}"] & T^\ast  X \arrow[r, "{{\mu}}"] & \gg^\ast 
 \end{tikzcd}, \qquad
{\mu}([g,\xi]) = \Ad^\ast _g \xi.
\end{equation}

\subsection{Hamiltonian geometry on $T^*X$}\label{subsechamiltonian}

Let us recall some results of Lisiecki \cite{Lisiecki87} which will be useful below.
Let $\gg_h^*$ denote the set of so-called hyperbolic elements in $\gg^*$, that is for a coadjoint orbit $\mathcal{O}_\lambda$, one has 
$$
\mathcal{O}_\lambda \subset \gg_h^* \,\, \mbox{iff} \,\, \mathcal{O}_\lambda \cap \ss^* \neq \emptyset.
$$
There is then a bijection of orbit spaces for the coadjoint action of $G$ and $K$
$$
\gg^*_h/G \stackrel{\sim}{\to} \ss^*/K,
$$
leading to an identification
$$
 \gg^*_h/G \ni \mathcal{O}_\lambda \stackrel{\sim}{\mapsto}   \lambda \in \bar \aa_+^* 
$$
where $\bar \aa_+^*$ denotes the closure of $\aa_+^*.$ Under this identification, regular coadjoint orbits (that is, of maximal dimension) correspond to $\lambda\in \aa_+^*$. We denote by $(\gg_h^*)_\mathrm{reg}\subset \gg_h^*$ the union of all regular coadjoint orbits.

One then obtains a $G$-equivariant diffeomorphism
$$
\aa_+^* \times G/MA \stackrel{\sim}{\to} (\gg_h^*)_\mathrm{reg}
$$
given by $(\lambda, g\cdot MA) \mapsto \mathrm{Ad}_g^*(\lambda).$ Note that, for each fixed $\lambda\in \aa_+^*$, this map gives a $G$-equivariant diffeomorphism between the regular coadjoint orbit $\mathcal{O}_\lambda$ and $G/MA$.

Recall from (\ref{e-moment-symm-space}) the moment map $\mu$ for the Hamiltonian $G$-action on $T^*X$. One obtains, see Paragraph (2.1) of \cite{Lisiecki87}, from 
$G$-equivariance of $\mu$ and from the natural identification $T_{x_o}^*X \cong \ss^*$, 
$$
\mu(T^*X) = \gg_h^*.
$$

Let $$(T^*X)_{\mathrm{reg}}:=\mu^{-1}((\gg_h^*)_\mathrm{reg})$$ denote the open dense subset of $T^\ast X$ where the moment map takes values in regular coadjoint orbits. Note that $(T^*X)_\mathrm{reg}$ is a multiplicity free Hamiltonian space  
with $G$-orbits being diffeomorphic to $G/M$.

Recall that each point in $(T^*X)_\mathrm{reg}$ can be uniquely represented as $[g,\lambda]\in G\times_K \ss^*$, for some $g\in G$ and $\lambda\in \aa_+^*$.
From Section 4 in \cite{Lisiecki87}, we then have a $G$-equivariant diffeomorphism
\begin{equation}\label{thebigdiffeo}
\varphi: (T^*X)_\mathrm{reg} \stackrel{\sim}{\longrightarrow} X\times \aa_+^*\times B,    
\end{equation}
given by
$$
\varphi([g,\lambda]) = (g\cdot K, \lambda, g\cdot MAN) = (g\cdot K, \lambda, k_1(g) \cdot M).
$$
We have $g\cdot MAN = k_1(g) \cdot M\in B$, where we recall the Iwasawa decomposition $g=k_1(g)e^{H(g)}n_1(g)$ in (\ref{iwasawas}).
The inverse of $\varphi$ can be used to pull back the symplectic structure from $(T^*X)_\mathrm{reg}$ to $X\times  \aa_+^*\times B$ so that Hamiltonian geometry and geometric quantization of $(T^*X)_\mathrm{reg}$ can be naturally described on $X\times \aa_+^*\times B.$

Let $A: X\times B\to \aa$ be defined by
$$
A(x,b) := A(k^{-1}g) = -H(g^{-1}k),
$$
where $x= g\cdot K, b=k\cdot M$ with $k\in K, g\in G.$ (See Theorem 4.2 in \cite{Lisiecki87}.) Below, we will still use the same notation
$$A:=A\circ \varphi_{X\times B} : (T^*X)_\mathrm{reg}\to \aa,$$ 
where $\varphi_{X\times B}$ is the component of $\varphi$ along $X\times B,$  so that 
\begin{equation}\label{defAxb}
A(x,b) = -H(g^{-1}k_1(g)) = H(g).    
\end{equation}

\subsection{The Fourier polarization for $T^\ast X$}\label{subsfourier}

Geometric quantizations of $T^* X$ preserving symmetries require $G$-invariant polarizations of the symplectic manifold $T^* X$. 

One such polarization, called the vertical polarization which will be denoted by $\mathcal{P}_0$,  is canonically given by the fibers of the cotangent bundle. The corresponding polarized sections are simply the functions factorizing through the bundle projection $\pi:T^* X\to X$, that is, functions on $X$. This leads to the Hilbert space of vertically polarized quantum states 
$$\mathcal{H}_{\mathcal{P}_0} \cong L^2(X,dx),$$ 
with $G$ acting by the regular representation.

The goal of this paper is to improve on the current understanding of the description of the (unitary) Fourier--Helgason transform 
$$L^2(X,dx)\to L^2\left(\mathfrak a_+^*\times B, \frac{d\lambda db}{\vert c(\lambda)\vert^2}\right),$$
where $c$ is the Harish-Chandra $c$-function 
(see \cite[Thm. III.1.5]{helgason:2008}), as an intertwining map between Hilbert spaces of quantum states induced by a change of polarization. 
For this purpose, one is interested in a ``Fourier polarization'' for which the space of leaves is measure-isomorphic to  $\mathfrak a_+^*\times B$. 

In \cite{Lisiecki87}, Lisiecki found such a polarization, which he called ``horizontal" (and which we will also call ``Fourier"). In the case when $G$ is complex, he proved that the Blattner-Kostant-Sternberg (BKS) pairing map between the Hilbert spaces for the quantizations in the vertical and horizontal polarizations is unitary and that it indeed coincides with the Fourier-Helgason transform. 

Our aim is to generalize this result to more general $G$, so that $G$ is  not necessarily complex, by using a ``quantum geodesic transform" instead of the BKS pairing. The idea is to connect the vertical and Fourier polarizations through a continuous family of polarizations generated by an appropriate Hamiltonian flow (in real time). The quantum geodesic transform is an analogue of the generalized coherent state transforms, which are associated to Hamiltonian flows in imaginary time, and which have been used in \cite{baier.hilgert.kaya.mourao.nunes:2023} and \cite{baier.ferreira.hilgert.mourao.nunes:2024} to describe families of quantizations for K\"ahler polarizations on cotangent bundles of compact Lie groups and of compact symmetric spaces, respectively.

Let us recall briefly some of the results of Lisiecki. Recall from (\ref{thebigdiffeo}) that the open dense subset $(T^*X)_\mathrm{reg}\subset T^*X$ is $G$-equivariantly diffeomorphic to $X\times  \aa_+^*\times B$ so that quantizations can be described on the latter.

\begin{proposition}[Lisiecki, Proposition (3.3) \cite{Lisiecki87}]\label{prop:Fourier-polarisation}
Let $$\widehat \pi := p_{\aa_+^*\times B} \circ \varphi \, : \, (T^*X)_\mathrm{reg} \longrightarrow \aa_+^* \times B,$$
where $p_{\aa_+^*\times B}:X\times  \aa_+^*\times B \to \aa_+^*\times B$ is the natural projection. 

The fibers of $\hat \pi$ define a Lagrangian foliation of $(T^*X)_\mathrm{reg}$. The corresponding polarization, $\mathcal{P}_\mathrm{F}$, is called the horizontal or Fourier polarization.
\end{proposition}

Note that on the symplectic manifold $X\times  \aa_+^*\times B$, whose symplectic structure is induced by $\varphi$ from the canonical symplectic structure on $(T^*X)_\mathrm{reg}$, the vertical and horizontal polarizations are simply given by the projections
$X\times \aa_+^*\times B\to X$ and $X\times \aa_+^*\times B\to \aa_+^*\times B,$ respectively. 

Since $G$ acts on $T^*X$ by symplectomorphisms, and since $\widehat\pi$ is $G$-equivariant, Proposition \ref{prop:Fourier-polarisation} can be proved by checking explicitly that the fiber of $\widehat\pi$ over $(\lambda,b_o), \lambda\in \aa_+^*$, is Lagrangian. (See \cite{Lisiecki87}.)

Consider now the function
\begin{eqnarray}\nonumber
S: X\times \aa_+^*\times B &\to& \mathbb{R}\\
\label{defS}
(x,\lambda, b) &\mapsto& \langle \lambda, A(x,b)\rangle,    
\end{eqnarray}
where $A(x,b)$ is given in (\ref{defAxb}). One obtains

\begin{theorem}[Lisiecki, Paragraph 6 \cite{Lisiecki87}]\label{lisieckithm} The Hilbert space of $\mathcal{P}_\mathrm{F}$-polarized quantum states is
$$
\mathcal{H}_{\mathcal{P}_F} \cong L^2\left (\aa_+^*\times B, \frac{d\lambda db}{\vert c(\lambda)\vert^2}\right),
$$   
where $\mathcal{P}_\mathrm{F}$-polarized functions on $(T^*X)_\mathrm{reg}\cong X\times \aa_+^*\times B$ have the form  $f e^{iS}$
for $f\in L^2\left (\aa_+^*\times B, \frac{d\lambda db}{\vert c(\lambda)\vert^2}\right)$ and $S$ in (\ref{defS}).
\end{theorem}

\section{Geodesic flow and a family of polarizations on $T^*X$}\label{sec3}

\subsection{Strategy}

Polarizations on symplectic manifolds  evolve naturally under push-forward by Hamiltonian flows. In the cases of cotangent bundles of compact Lie groups, or of symmetric spaces of compact type, as well as for other families of symplectic manifolds, in fact, Hamiltonian flows generated by Casimir functions on the Lie algebra, can be used to describe very interesting continuous families of quantizations. In these examples, the analytic continuation to imaginary time of these Hamiltonian flows allows for a continuous interpolation between the quantization in the vertical polarization and interesting polarizations of mixed type, with rich geometric properties and relations to representation theory. 
In this way, one obtains, for example, a geometric interpretation of the Peter-Weyl theorem by relating it to the Borel-Weil realization of irreducible representations of a compact Lie group in terms of holomorphic line bundles over coadjoint orbits. In these examples, the path of polarizations, induced by the Hamiltonian flow in imaginary time, corresponds to a geodesic in the space of K\"ahler metrics, with the mixed polarization corresponding to the limit of infinite geodesic time. Over this path of polarizations one can consider the corresponding ``quantum" bundle of Hilbert spaces of polarized (or quantum)  states. The so-called generalized coherent state transform, of which the Segal-Bargmann transform is the paradigmatic example, provides the natural intertwining map between the Hilbert spaces for the different quantizations along the family.\footnote{In some examples, the gCST coincides with the BKS pairing map but this is not usually the case.}
(See \cite{baier.hilgert.kaya.mourao.nunes:2023, baier.ferreira.hilgert.mourao.nunes:2024}.) 

In the present case, as will be described below, due to the fact that $X=G/K$ is of noncompact type, this strategy needs to be adapted. In fact, in the present case, one needs to consider Hamiltonian flows in real time. We will consider, then, the Hamiltonian flow generated by the quadratic Casimir 
\begin{equation}\label{hamil}
    h: T^\ast X \longrightarrow \mathbb{R}, \quad [g,\xi]\mapsto \frac{1}{2}\kappa(\xi,\xi).
\end{equation}
Note that
$$\kappa(\xi,\xi)=\kappa(\Ad^*_g(\xi),\Ad^*_g(\xi)) =
\kappa(\mu([g,\xi]),\mu([g,\xi])) \geq 0,$$
where $\mu$ is the moment map for the Hamiltonian $G$-action in (\ref{e-moment-symm-space}). This coincides with the usual ``energy function" defined on $T^*M$, for a Riemannian manifold $M$, and whose Hamiltonian flow (with respect to the canonical symplectic structure on $T^*M$) coincides with the geodesic flow (under the identification $TM\cong T^*M$ induced by the Riemannian structure).

Therefore, the Hamiltonian flow of the Hamiltonian vector field of $h$, $X_h$, is the geodesic flow for the $G$-invariant metric on $X$. It is given by 
\begin{equation}\label{hamflow}
    \varphi_t^{X_h}([g,\xi]) = [g\mathrm{e}^{t\Tilde{\xi}}, \xi],
\end{equation}
where $\Tilde{\xi}\in \mathfrak{g}$ is the image of $\xi \in \mathfrak{g}^\ast$ under the isomorphism given by the Killing form.

The geodesic flow on $(T^*X)_\mathrm{reg}$ becomes
\begin{equation}
\label{e_15}
 \eta_t:=   \varphi_t^{X_h}([g,\lambda_H]) = [g\mathrm{e}^{t{H}}, \lambda_H],  \, t\in \mathbb{R},
\end{equation}
with $H \in \mathfrak{a}_+$ and where $\lambda_H \in \aa_+^*$ is the corresponding element under the identification $\gg \cong \gg^*$, using the Killing form so that the map $H \mapsto \lambda_H $ is the inverse of $\xi \mapsto \Tilde \xi$, $\Tilde \lambda_H = H$. 

\begin{remark}
    {Recall that, associated with the inclusions $\aa\hookrightarrow \ss\hookrightarrow\gg\hookleftarrow\kk$ we have the restriction maps $\aa^*\leftarrow \ss^*\leftarrow\gg^*\rightarrow\kk^*$. Using the Killing form $\kappa$ , we get on each level the isomorphism $\mathfrak h^*\rightarrow \mathfrak h, \xi\mapsto\tilde \xi=H_\xi$ with inverse $\mathfrak h\rightarrow \mathfrak h^*, H\mapsto\lambda_H$ via
$$\kappa(\tilde\xi,H)=\kappa(H_{\xi},H)= \xi(H)\quad\text{and}\quad \kappa(H,H')=\lambda_H(H') .$$ 
These isomorphisms yield splittings for the restrictions, so we get injections $\aa^*\hookrightarrow \ss^*\hookrightarrow\gg^*\hookleftarrow\kk^*$ commuting with the isomorphisms.}
\end{remark}

We set  
$$x_t := \pi (\eta_t) = ge^{t{H}}\cdot  K\in X. $$ 

Below, we will first show that the geodesic flow continuously interpolates between the vertical polarization $\mathcal{P}_0$, at $t=0$, and the Fourier polarization $\mathcal{P}_\mathrm{F}$ of Proposition \ref{prop:Fourier-polarisation} in the limit $t\to +\infty$. We will then use a real-time analogue of the generalized coherent state transform, which we call the ``quantum geodesic transform", to lift the flow to the bundle of Hilbert spaces of quantum states, showing that in the limit $t\to +\infty$ we obtain the unitary Fourier--Helgason transform.

\subsection{Geodesic flow from the vertical to the Fourier polarization}
\label{sectionevolutionofpol}

Let us now study the evolution of the vertical polarization under the Hamiltonian flow generated by $h$.

A $\mathcal{P}_0$-polarized function is a function on $X$ (pulled-back to $T^\ast X$)
\begin{equation*}
    (\pi^\ast\, f)([g,\xi]) = f(g\cdot K),\, f\in C^\infty (X).
\end{equation*}
Note that, from the definition of the canonical symplectic form on $T^*X,$ the polarization $\mathcal{P}_0$ is generated by the Hamiltonian vector fields of smooth $\mathcal{P}_0$-polarized functions. Then we have

\begin{proposition}\label{timeevolution}
    On $(T^*X)_\mathrm{reg}$, the polarization $\mathcal{P}_t:= (\varphi_{-t}^{X_h})_\ast\mathcal{P}_0$ is generated by the functions 
    \begin{equation*}
        F_t([g,\lambda_H]) := f(g\mathrm{e}^{tH}\cdot K) , \, \, f\in C^\infty(X).
    \end{equation*}
\end{proposition}
\begin{proof}
    We have that 
\begin{align*}      ((\varphi_t^{X_h})^\ast\pi^\ast f)([g,\lambda_H])  &=  f( \pi (\varphi_t^{X_h}([g,\lambda_H])))  = f(\pi([g\mathrm{e}^{tH}, \lambda_H]))
          = f(g\mathrm{e}^{tH}\cdot K).
\end{align*}

\end{proof}

Let us now show that  when $t\to +\infty$, with the limit being taken pointwise in the Lagrangian Grassmannian bundle on $(T^*X)_\mathrm{reg}$, there exists a limit polarization and that it coincides with the Fourier polarization in Proposition \ref{prop:Fourier-polarisation}, that is
\begin{equation}\label{eq:P-Pl}
    \exists \lim_{t\rightarrow\infty}\mathcal{P}_t = \mathcal{P}_\mathrm{F}.
\end{equation}

Recall that a regular coadjoint orbit $\mathcal{O}_\lambda, \lambda\in \aa_+^*$, is a symplectic manifold with canonical symplectic structure given by the Kostant--Kirillov--Souriau form. One defines a real polarization $\mathcal{P}_{\mathcal{O}_\lambda}$ on $\mathcal{O}_\lambda$ to be the $G$-invariant real polarization of $\mathcal{O}_\lambda$ defined by the fibers of the canonical projection 
$$
\mathcal{O}_\lambda \cong G/MA \to B=G/MAN,
$$
as in Paragraph 3 in \cite{Lisiecki87}. More precisely, for $\lambda\in \aa_+^*,$ let
$$
\alpha_\lambda : G/MA \to \mathcal{O}_\lambda
$$
be the $G$-equivariant diffeomorphism given by
\begin{equation}\label{alphalambda}
\alpha_\lambda (g\cdot MA) := \mathrm{Ad}_g^* \lambda.    
\end{equation}

From the Bruhat decomposition $\gg = \bar\nn + \mm + \aa+\nn,$ we have an identification 
$$
T_{e\cdot MA}(G/MA) \cong \bar\nn + \nn.
$$
Let $\mathcal{D}\subset T(G/MA)$ be the half-dimensional $G$-invariant distribution given at the identity coset by 
    $$
    \mathcal{D}_{e\cdot MA} := \nn.
    $$
Denote by $\overline{\mathcal{D}}\subset T(G/MA)$ the complementary $G$-invariant distribution such that 
$$
    \overline{\mathcal{D}}_{e\cdot MA} := \bar\nn.
$$

Then the polarization $\mathcal{P}_{\mathcal{O}_\lambda}$ can be obtained as follows.

\begin{lemma}\label{lemmadistribution}
    For $\lambda\in \aa_+^*$
    $$
    P_{\mathcal{O}_\lambda} = (\alpha_\lambda)_* \mathcal{D}.
    $$
\end{lemma}

\begin{proof}
From equation (3.1.2) in \cite{Lisiecki87}, we have that $\mathcal{P}_{\mathcal{O}_\lambda}$ is the $G$-invariant polarization on $\mathcal{O}_\lambda$ which at $\lambda\in \mathcal{O}_\lambda$ is given by 
the subspace 
$$
(\mm+\aa+\nn)^\perp = \left\{ \xi\in \gg^*\mid \forall \tilde\mu\in \mm+\aa+\nn :\  \xi (\tilde \mu)=0 \right\}\subset \gg^*.
$$ 
Let $\tilde\xi\in \nn, \epsilon >0$, and consider the curve $\gamma(t) = e^{t\tilde\xi}\cdot MA, \, t\in (-\epsilon,\epsilon)$, through the identity coset in $G/MA$. Then,
$$
\langle \frac{d}{dt}_{\vert_{t=0}} \mathrm{Ad^*_{e^{t\tilde \xi}}\lambda, \tilde\mu}\rangle =  \langle \lambda, \frac{d}{dt}_{\vert_{t=0}}\mathrm{Ad}_{e^{-t\tilde\xi}}\tilde\mu\rangle = \langle \lambda, -[\tilde\xi, \tilde \mu]\rangle,\, \tilde\mu\in \gg.
$$   
Since $\tilde \xi\in \nn,$ this is zero if $\tilde\mu\in \mm+\aa+\nn,$ so that at $e\cdot MA$ indeed
$$
(\alpha_\lambda)_* \nn = (\mm+\aa+\nn)^\perp.
$$
\end{proof}

Therefore, for every $\lambda\in \aa_+^*,$ the polarization $\mathcal{P}_{\mathcal{O}_\lambda}$ on $\mathcal{O}_\lambda$ can be obtained by push-forward from a fixed $G$-invariant distribution on $G/MA$.

Let $\hat\aa_+^*$ be the set of $\nu\in \aa_+^*$ which occur as highest weights of finite dimensional irreducible spherical representations of $G$ in the sense of \cite[Thm.~V.4.1]{helgason:1984}. 
Note that the linear span of $\hat\aa_+^*$ gives all of $\aa^*.$
We denote the representation of $G$ corresponding to $\nu$ by $(\pi_\nu,\mathcal{H}_\nu)$ and fix a highest weight vector $v_\nu\in\mathcal H_\nu$. The derived representation of $\pi_\nu$ extends to $\gg_\CC$, which in turn integrates to a representation $\hat \pi_\nu$ of the simply connected complex Lie group $\hat G_\CC$ with Lie algebra $\gg_\CC$. We may assume that this representation is unitary on the maximal compact subgroup $U$ of $\hat G_\CC$ with Lie algebra $(\kk+i\ss)$ . Thus, the derived representation is skew-Hermitian on $\kk+i\ss$. Hence, the restriction of $\pi_\nu$ from $G$ to $K$ is unitary.

\begin{lemma}\label{lem:highest-weight-equiv}
    Let $\nu\in \hat \aa_+^*$. For any $v\in \mathcal{H}_\nu$, the functions on $G$
    \begin{equation}
        H_v^\nu(g) := \langle v, \pi_{\nu}(g)v_\nu \rangle \label{eq:func_H}
    \end{equation}
 satisfy the following properties:
 \begin{itemize}
     \item[(i)] They are $A$-equivariant, that is  for all $a\in \aa$, 
    $$\langle v, \pi_\nu(ge^a)v_\nu\rangle = e^{\langle \nu,a \rangle}\langle v, \pi_\nu(g)v_\nu \rangle$$ 
     \item[(ii)] They are $MN$-invariant, that is, for all $m\in M, n\in N$
     $$\langle v, \pi_\nu(gmn)v_\nu\rangle = \langle v, \pi_\nu(g)v_\nu\rangle. $$
 \end{itemize}
\end{lemma}

\begin{proof}
The only part that is not immediate is the $M$-invariance, but this is guaranteed by \cite[Thm.~V.4.1(i)]{helgason:1984}. 
\end{proof}

\begin{remark}
    From (i) in Lemma \ref{lem:highest-weight-equiv}, we conclude that the functions $H^\nu_v$ in \eqref{eq:func_H} define sections of a line bundle
    $$L_\nu := (G\times_{A} \mathbb{C}\cdot v_\nu) \longrightarrow  G/A.$$
    Point (ii) in Lemma \ref{lem:highest-weight-equiv} then gives that $L_\nu$ is actually a pull-back of a line bundle 
    $$
    \tilde L_\nu:= (G\times_{MAN} \mathbb{C}\cdot v_\nu) \longrightarrow B, 
    $$
    and that the functions $H_v^\nu$ are actually pull-backs of sections of $\tilde L_\nu$ under the projection $G/A\to B=G/MAN$. 
\end{remark}

\begin{lemma}\label{lemma_existsnu}
    For each non-zero $\tilde \xi\in \bar\nn$ and $\nu\in \hat\aa_+^*$ we have $\pi_\nu(\tilde\xi) v_\nu\neq 0.$ 
\end{lemma}

\begin{proof}

Lemma~\ref{lem:highest-weight-equiv} implies that, for $\nu\in \hat\aa_+^*$, $\mm+\nn$ is in the annihilator of $v_\nu$ and $\aa$ is in the stabilizer of  $\CC v_\nu$.  Thus the stabilizer of $\CC v_\nu$ contains the minimal parabolic subalgebra $\mm+\aa+\nn$. Now \cite[\S~VII.7]{knapp:1996} yields that the stabilizer of $\CC v_\nu$ is a real parabolic subalgebra of $\gg$ of the form $\mm+\aa+\nn+\bar\uu$ with $\uu$ a subalgebra of $\nn$. Any element $\tilde\xi\in \bar\nn$ stabilizing $\CC v_\nu$ has to annihilate it because any $\gg_\alpha$ maps $v_\nu$ into the weight space of $\nu+\alpha$. Thus the annihilator of $v_\nu$ contains $\mm+\nn+\bar\uu$. If $\bar\uu\not=0$, then it contains a root vector $\tilde\zeta\in \gg_{-\alpha}$ with $\alpha\in \Sigma^+$.  By \cite[Prop.~6.52]{knapp:1996} we have 
$$cH_\alpha=[\tilde\zeta,\sigma\tilde\zeta]\in (\mm+\aa)\cap \ss=\aa$$
with $c\neq 0$. But $\pi_\nu(\sigma\tilde\zeta)v_\nu=0$ as $\sigma\tilde\zeta\in \gg_\alpha$, so we also have $\pi_\nu([\tilde\zeta,\sigma\tilde\zeta])v_\nu=0$. But then Lemma~\ref{lem:highest-weight-equiv}(i) says that $\langle \nu, [\tilde\zeta,\sigma\tilde\zeta]\rangle =0 $ which contradicts $\nu\in \aa_+^*$. Thus, we have shown that $\bar\uu=0$.

We now know that only elements of $\mm+\aa+\nn$ can stabilize $\CC v_\nu$. In particular, no element of $\bar\nn$ can stabilize $\CC v_\nu$, let alone annihilate it.

\end{proof}

\begin{lemma}\label{polarizationoncoadjointorbit} 
 Whenever well-defined, on open sets $U\subset G/MA$, consider the functions 
$${F}^\nu_{v,\Tilde{v}}: U \to \mathbb{C}$$ given by 
 \begin{equation}\label{eq:Pl-remaining}
        {F}^\nu_{v,\Tilde{v}}(g \cdot MA) := \frac{\langle v, \pi_\nu(g)v_\nu\rangle}{\langle \Tilde{v}, \pi_\nu(g)v_\nu\rangle}.
    \end{equation}

Then
\begin{itemize}
    \item[(i)]  ${F}^\nu_{v,\Tilde{v}}$ is constant along the distribution $\mathcal{D}\subset T(G/MA)$ of Lemma \ref{lemmadistribution}.
    \item[(ii)] For each non-zero tangent vector $u\in \overline{\mathcal{D}}_{g\cdot MA}$, there exists $\nu\in \hat\aa_+^*, v, \tilde v\in \mathcal{H}_\nu$, such that 
    $$
    dF^\nu_{v,\tilde v} (u) \neq 0.
    $$
    \item[(iii)] For $\lambda\in \aa_+^*$, the Hamiltonian vector fields of the set of functions 
    $(\alpha_\lambda^{-1})^* {F}^\nu_{v,\Tilde{v}}$ generate the polarization $\mathcal{P}_{\mathcal{O}_\lambda}$. 
\end{itemize}
\end{lemma}

\begin{proof}
Let $\tilde \xi\in \nn$ and consider the curve $\gamma(t)= ge^{t\tilde \xi }\cdot MA, t\in (-\epsilon, \epsilon), \epsilon>0$, so that $\dot\gamma(0) \in \mathcal{D}_{g\cdot MA}.$
To prove (i), note that, since $\pi_\nu (\tilde \xi) v_\nu=0$, a simple explicit calculation gives that
$$
\frac{d}{dt}_{\vert_{t=0}} {F}^\nu_{v,\Tilde{v}} (\gamma(t)) = 0.
$$

To prove (ii), let $u=\dot\gamma(0)$, where $\gamma(t) := ge^{t\tilde \xi}$ with $\tilde \xi\in \bar\nn$. We have
$$
 d {F}^\nu_{v,\Tilde{v}} (u) \, =  \, \frac{\langle v, \pi_\nu(g)\pi_\nu(\tilde \xi) v_\nu\rangle \langle \tilde v, \pi_\nu(g) v_\nu\rangle - \langle v, \pi_\nu(g)v_\nu\rangle \langle \tilde v, \pi_\nu(g)\pi_\nu(\tilde\xi)v_\nu\rangle}{\langle \tilde v, \pi_\nu(g) v_\nu\rangle^{2}}   \, .
$$
From Lemma \ref{lemma_existsnu}, for each such $\tilde \xi\in \bar\nn$, one can find $\nu\in \hat\aa_+^*$ such that $\pi_\nu(\tilde\xi) v_\nu\neq 0.$ 

Let then
$v= \pi_\nu(g) \pi_\nu(\tilde \xi) v_\nu$ and $\tilde v = \pi_\nu(g) v_\nu.$ Then,
$$
dF^\nu_{v,\tilde v} (u) = \vert\vert \tilde v\vert\vert^{-4} \left( \vert\vert v \vert\vert^2 \vert\vert \tilde v\vert\vert^2 - \vert \langle v, \tilde v\rangle\vert^2\right).
$$
The numerator vanishes iff $\pi_\nu (\tilde\xi) v_\nu$ is proportional to $v_\nu$ which will not happen if $0\not=\tilde \xi \in \bar\nn$. 

Given the description of $\mathcal{P}_{\mathcal{O}_\lambda}$ in Lemma \ref{lemmadistribution}, one immediately obtains that (i) and (ii) imply (iii). Note that the differential of 
$(\alpha_\lambda^{-1})^* {F}^\nu_{v,\Tilde{v}}$ vanishes along vectors lying in the distribution $\mathcal{P}_{\mathcal{O}_\lambda}$, so that the Hamiltonian vector field of $(\alpha_\lambda^{-1})^* {F}^\nu_{v,\Tilde{v}}$ is in $\mathcal{P}_{\mathcal{O}_\lambda}$. 
\end{proof}

Let $\mu^{\mathrm{Kir}}: (T^*X)_\mathrm{reg}\to \aa_+^*$ be the Kirwan moment map which is obtained by rotating the values of the moment map 
$\mu$ in (\ref{e-moment-symm-space}) to the positive Weyl chamber.   Let $$r:=\mathrm{rank}(G/K)=\mathrm{dim}\,\mu^\mathrm{Kir}(T^\ast X)_\mathrm{reg}=\dim A.$$

\begin{proposition}\label{lemmpropertiesP_F}
The Fourier polarization $\mathcal{P}_\mathrm{F}$ on $(T^*X)_\mathrm{reg}$ is uniquely defined by the following two properties
\begin{enumerate}
    \item[(i)] The components of $\mu^\mathrm{Kir}$ are $\mathcal{P}_\mathrm{F}$-\hspace{0pt}polarized, that is, their Hamiltonian vector fields define an $r$-\hspace{0pt}dimensional subspace of the directions in 
    $\mathcal{P}_\mathrm{F}$.
       \item[(ii)] Consider 
    $$\tilde\mu: (T^*X)_\mathrm{reg}\to G/MA, \quad [g,\xi]\mapsto \alpha_{\mu^\mathrm{Kir}([g,\xi])}^{-1}\circ\mu([g,\xi]).$$
The pull-backs under $\tilde\mu$ of  functions factorizing through $G/MA\to G/MAN$ are $\mathcal{P}_\mathrm{F}$-\hspace{0pt}polarized, that is, their Hamiltonian vector fields generate a subspace of $\mathcal{P}_\mathrm{F}$ which complements the subspace given in (i).
\end{enumerate}
\end{proposition}

\begin{proof}
   \begin{enumerate}
       \item[(i)] Recall that, under the identification $(T^*X)_\mathrm{reg} \cong X\times  \aa_+^*\times B$ given by the map $\varphi$ in (\ref{thebigdiffeo}), the Fourier polarization $\mathcal{P}_\mathrm{F}$ is defined by the projection $X\times \aa_+^*\times B \to  \aa_+^*\times B$. The coordinates of $\mu^{\mathrm{Kir}}$ are then given directly by the component along $\aa_+^*$
    that is, under the above identification $\mu^\mathrm{Kir} (x,\lambda,b)= \lambda.$ This is clearly constant along the fibers of the projection defining $\mathcal{P}_\mathrm{F},$ which proves (i).

    \item[(ii)]  
    Note that the projection defining $\mathcal{P}_\mathrm{F}$, $(T^*X)_\mathrm{reg}\to \aa_+^*\times B$,  factorizes through
   $\mu:(T^*X)_\mathrm{reg}  \to (\gg^*_h)_\mathrm{reg} \cong \aa_+^*\times G/MA$, with the projection onto $B$ being the projection $\mathcal{O}_\lambda \cong G/MA \to B$ which defines $\mathcal{P}_{\mathcal{O}_\lambda},$ as follows from Lemmas \ref{lemmadistribution}, \ref{lem:highest-weight-equiv} and \ref{polarizationoncoadjointorbit}. 
    As $\tilde \mu$ is the composition of $\mu$ with the projection onto the second factor if we use the identification $(\gg^*_h)_\mathrm{reg} \cong \aa_+^*\times G/MA$ in the description of 
$$\mu: (T^*X)_\mathrm{reg}\to (\gg^*_h)_\mathrm{reg},$$ then (ii) follows.
   \end{enumerate}  
\end{proof}

We are now ready to prove our main result \eqref{eq:P-Pl} of this section.

\begin{theorem}\label{thmpolarization}
    The limit polarization on $(T^\ast X)_\mathrm{reg}$ in \eqref{eq:P-Pl} exists and it coincides with the Fourier polarization,
    $$
\lim_{t \to \infty}    \mathcal{P}_t = \lim_{t \to \infty}  \,  (\varphi_{-t}^{X_h})_\ast\mathcal{P}_0 = 
  \mathcal{P}_\mathrm{F}.  $$

           \end{theorem}

\begin{proof}For $\nu\in \hat\aa_+^*$, a  spherical vector $v_\nu^K\in\mathcal{H}_\nu$ and $v\in \mathcal{H}_\nu$, consider the  
functional $v^*$ on $\mathcal{H}_\nu$ defined by $v^*(w):=\langle v,w\rangle:=\langle v,w\rangle_{\mathcal{H}_\nu}$ and the rank one operator 
$$v_\nu^K \otimes v^\ast:\mathcal{H}_\nu\to\mathcal{H}_\nu,\quad w\mapsto \langle w,v\rangle_{\mathcal{H}_\nu} v_\nu^K.$$
Then the function $f_{\nu,v}: (T^*X)_\mathrm{reg}\to \CC$ defined by
\begin{equation}
    f_{\nu,v}([g,\xi]) = \mathrm{tr}(\pi_\nu(g) (v_\nu^K \otimes v^\ast)),\quad g\in G, \xi\in \aa_+^*
\end{equation}
is vertically polarized, so its Hamiltonian vector field lies in $\mathcal{P}_0$.

Let $\left\{w_j\right\}_{j=1, \dots \dim \mathcal{H}_\nu},$ be a basis of weight vectors for $\mathcal{H}_\nu$, where $w_j$ corresponds to the weight 
$\nu_j = \nu - \beta_j$ where $\beta_j$ is a linear combination of positive roots with nonnegative coefficients. Let us fix $w_1=v_\nu$ so that $\beta_1=0.$
From Proposition \ref{timeevolution}, the time $t$ evolution  of $f_{\nu,v}$ under the geodesic flow is
\begin{equation*}
\begin{array}{ll}
  f^t_{\nu, v}([g,\xi])   & :=  \mathrm{tr}(\pi_\nu(g)\pi_\nu(\mathrm{e}^{t\Tilde{\xi}})\, v_\nu^K \otimes v^\ast) \medskip\\
     & = \mathrm{tr}\left(\pi_\nu(g)\pi_\nu(\mathrm{e}^{t\Tilde{\xi}})\, \sum_{j=1}^{\dim \mathcal{H}_\nu} w_j^* (v_\nu^K) w_j\otimes  v^\ast\right) \medskip\\
     & = \sum_{j=1}^{\dim \mathcal{H}_\nu} \mathrm{e}^{t(\nu-\beta_j)(\tilde \xi)}\mathrm{tr}(\pi_\nu(g) w_j^* ( v_\nu^K) w_j\otimes  v^\ast) \medskip\\
     & = \mathrm{e}^{t\nu(\tilde \xi)}(\mathrm{tr}(\pi_\nu(g)\circ(v_\nu \otimes v_\nu^\ast)\circ (v_\nu^K \otimes v^\ast)))+ \, {o(e^{t\nu(\tilde \xi)})}, 
\end{array}
\end{equation*}
where for the last equality we have used that $\tilde\xi\in \aa_+$ and $\nu\in \aa_+^*$ so that $\nu(\tilde\xi)=\langle\nu,\tilde\xi\rangle> \langle\nu-\beta_j,\tilde\xi\rangle$, for $j>1$. Thus we have
\begin{equation}\label{eq:lot}
    f^t_{\nu, v}([g,\xi]) =   \mathrm{e}^{t\nu(\tilde\xi)}\langle v_\nu, v_\nu^K \rangle \langle v, \pi_\nu (g)v_\nu \rangle + o(e^{t\nu(\tilde \xi)}),
\end{equation}
and find  
\begin{equation*}
\hat F^\nu_{v,\Tilde{v}}([g, \xi]):=
\lim_{t\rightarrow\infty} \frac{f^t_{\nu,v}([g,\xi])}{f^t_{\nu,\Tilde{v}}([g, \xi])} = \frac{\langle v, \pi_\nu(g) v_\nu\rangle}{\langle \Tilde{v}, \pi_\nu(g) v_\nu\rangle}\overset{\eqref{eq:Pl-remaining}}{=} {F}^\nu_{v,\Tilde{v}}(g \cdot MA).   
\end{equation*}
Also, away from zeros of $f_{\nu,v}^t$, for any choice of branch for $\log$, we have
\begin{eqnarray*}
   \lim_{t \to \infty} \frac{1}{t} \,  \log f_{\nu,v}^t ([g,\xi]) &=&  
   \nu (\tilde \xi) +  \lim_{t \to \infty} \, \frac 1t \, \log\left(\langle v_\nu, v_\nu^K \rangle \langle v, \pi_\nu (g)v_\nu \rangle + o(1) \right) 
  \ =\ \nu (\tilde \xi).
\end{eqnarray*}
These limits imply the existence of the  limit polarization,
$\lim_{t\rightarrow\infty}\mathcal{P}_t$ in \eqref{eq:P-Pl} and show that it is generated by the Hamiltonian vector fields of the components of the Kirwan moment map
\begin{equation*}
    \mu^{\mathrm{Kir}}([g,\xi]) = \xi, \,\, \xi\in \aa_+^*, 
\end{equation*} together with the Hamiltonian vector fields of
$\hat F^\nu_{v,\Tilde{v}}$.

In view of \eqref{alphalambda}  we have
$$
\hat F^\nu_{v,\Tilde{v}}([g,\xi]) 
=F^\nu_{v,\Tilde{v}} \circ (\alpha_{\xi})^{-1} \circ \mu([g,\xi]) = {F}^\nu_{v,\Tilde{v}}\circ \tilde \mu([g,\xi]).
$$
Moreover, by Lemma~\ref{lem:highest-weight-equiv}(ii) the function ${F}^\nu_{v,\Tilde{v}}:G/MA\to \CC$ factors through $G/MA\to G/MAN$. Thus, Proposition \ref{lemmpropertiesP_F}(ii) implies that 
$\hat F^\nu_{v,\Tilde{v}}$ is $\mathcal{P}_\mathrm{F}$-\hspace{0pt}polarized. This shows that the limit polarization is contained in and hence, by Lemma \ref{polarizationoncoadjointorbit}, it is equal to $\mathcal{P}_\mathrm{F}$. 
    \end{proof}

\subsection{Evolution of the canonical form along the geodesic flow}\label{sec:Evolution}

In this section, we will describe how the canonical form for the vertical polarization $\mathcal{P}_0$ evolves along the geodesic flow.
This will be important later for the consideration of the crucial contribution of the half-form correction in the geometric quantization description of the Fourier--Helgason transform.

Recall that a canonical form associated to a polarization $\mathcal{P}$ on a symplectic manifold $(M^{2n},\omega)$ is 
a {nowhere-vanishing} $n$-form annihilated by $\mathcal{P}.$ 
A canonical form
for the polarization $\mathcal{P}_t, t\geq 0$, described in Proposition  \ref{timeevolution},
is given by
\begin{equation}\label{detat}d\eta_t := \left(\varphi_t^{X_h}\right)^* \circ  \pi^* dx,\end{equation}
where $dx$ is the $G$-invariant $\dim(X)$-form on $X$ associated with the measure $dx$ introduced in Subsection~\ref{Notation}. Note that the canonical form for the vertical polarization 
$\mathcal{P}_0$ is $d\eta_0 = \pi^* dx.$

The main result in this section is Proposition \ref{prop:canonical-form-lift}, below, which describes the asymptotic behavior of $d\eta_t$ as $t\to +\infty$. We will consider some preparatory material before stating and proving it.

Recall the Lie algebra notations introduced in Subsection~\ref{Notation}.
We will denote generators for $\nn$ given by root vectors $\left\{E_\alpha\right\}_{\alpha\in \Sigma^+},$ where the positive restricted roots will always be counted with multiplicities in what follows. \footnote{For simplicity of notation, $E_\alpha$ will stand really for 
$\left\{E_{\alpha,j}\right\}_{j=1, \dots ,m_\alpha}$ where $m_\alpha$ is the multiplicity of $\alpha$.} We consider the root vectors
$$
E_{-\alpha} := \sigma (E_\alpha), \quad \alpha\in \Sigma^+ ,
$$
where $\sigma$ is the Cartan involution, as the generators of 
$\bar\nn $. We have
$$
\kk = \mm \oplus \bigoplus_{\alpha\in \Sigma^+} \kk_\alpha
\quad\text{and}\quad
\ss = \aa \oplus \bigoplus_{\alpha\in \Sigma^+} \ss_\alpha,
$$
where $\kk_\alpha := \kk \cap (\gg_\alpha + \gg_{-\alpha})$ and $\ss_\alpha := \ss \cap (\gg_\alpha + \gg_{-\alpha})$.
We will consider the generators 
$$
E^K_\alpha := E_\alpha + E_{-\alpha}
\quad\text{and}\quad 
E^{\ss}_\alpha := E_\alpha - E_{-\alpha}, \quad \alpha \in \Sigma^+,
$$
of $\kk_\alpha$ and $\ss_\alpha$ respectively. For $\mm$ we take generators  $\left\{M_i\right\}_{i=1, \dots, m}$, where $m:=\dim \mm.$ Let us recall a basic result.

Let $\Theta$ be the (left-invariant) $\gg$-valued  Maurer-Cartan form on $G$, so that 
$$
\Theta_g (g\cdot Y) = Y,\,\, Y\in T_eG
$$
where we will denote by $g\cdot Y$ the push-forward of $Y\in T_eG$ by the left-translation $L_g:G\to G, \, g\in G$.

\begin{lemma}\label{lem:MC-right-transl}
 Under right translation by $g\in G$, $R_g:G\to G$, $R_g u = ug, u\in G,$ we have
$$
R_g^* \Theta = \mathrm{Ad}_{g^{-1}} \Theta.
$$
\end{lemma}

\begin{proof}
    The proof is standard, but let us include it for completeness. Let $Y\in T_eG$. Then
    $$
    (R_g^*\Theta)_{u} (u\cdot Y) = \Theta_{ug} (u \cdot Y \cdot g) = \Theta_{ug} (ug\cdot \mathrm{Ad}_{g^{-1}}Y) = \mathrm{Ad}_{g^{-1}} Y,
    $$
    where we denote by $Y\cdot g$ the push-forward of $Y\in T_eG$ by  $R_g$.
\end{proof}

In order to prove Proposition \ref{prop:canonical-form-lift} it will be useful to recall the diffeomorphism 
$$
(T^*X)_\mathrm{reg} \cong G/M \times \aa_+^*,
$$
which assigns to the point $[g,\xi]\in (T^*X)_\mathrm{reg}, g\in G, \xi\in \aa_+^*$, the pair $(g\cdot M, \xi)\in G/M\times \aa_+^*$, cf. \eqref{thebigdiffeo} and the diffeomorphism $G/M\cong G/K\times K/M$. Lift the geodesic flow given in \eqref{hamflow} to $G\times \aa_+^*$ through the natural projection $G\times \aa_+^* \to G/M \times \aa_+^*$. This means, for $t\in \RR$, we let $\Phi_t:G\times \aa_+^* \to G\times \aa_+^*$ be the map
    $$
    \Phi_t (g,\xi) := (g e^{t\tilde \xi}, \xi),
    $$
    where, as before, $\tilde \xi\in \aa_+$ is the image of $\xi\in \aa_+^*$ under the identification $\gg^*\cong \gg.$ We will denote also by $\Theta$ the pull-back of $\Theta$ from $G$ by the natural projection $$p_G:G\times \aa_+^* \to G.$$

We also consider a Killing orthonormal basis  $\left\{A_j\right\}_{j=1, \dots, r}$ for $\aa$ 
and let  $\left\{\hat \eta^A_j\right\}_{j=1, \dots r}$ denote the translation invariant $1$-forms on $\aa^*$ such that $\hat\eta_j^A(\xi) = \tilde \xi_j$ where $\tilde \xi =: \sum_{j=1}^r \tilde \xi_j A_j$. 
If we equip $\aa^*$ with the dual metric of the Killing form and write $\left\{\rho_j\right\}_{j=1, \dots, r}$ for the (orthonormal) dual basis of $\left\{A_j\right\}_{j=1, \dots, r}$, then the $\hat\eta_j^A\in {\aa^*}^*=\aa$ form the dual basis of $\left\{\rho_j\right\}_{j=1, \dots, r}$, whence
\begin{equation}\label{lettherebedlambda}
d\mathrm{vol}_{\aa^*} = \bigwedge_{j=1}^r \hat \eta_j^A,
\end{equation}
is the volume form on the Riemannian manifold $\aa^*$.

According to the normalization of measures in Chapter II, Paragraph 3, Section 1 of \cite{helgason:2008}, we have
$$
d\lambda = (2\pi)^{-\frac{r}{2}} d\mathrm{vol}_{\aa^*},
$$
as mentioned at the end of Section \ref{Notation}. 
We will also use the same notation $\hat\eta^A_j$ for the pull-back of this one-form to $G\times \aa_+^*$ by 
the natural projection $p_{\aa_+^*}: G\times \aa_+^* \to \aa_+^*.$
    
\begin{lemma} \label{lem:MC-flow-pullback}
We have
\begin{equation}\label{pull-backedtheta}
(\Phi_t^* \Theta)_{(g,\xi)} = \mathrm{Ad}_{e^{-t\tilde \xi}} \Theta + t\sum_{j=1}^r \hat \eta^A_j A_j.
\end{equation}
\end{lemma}

\begin{proof}
Let $(Y,\xi')$ be a tangent vector to $G\times \aa_+^*$ at the point $(e,\xi).$ 
    We have
\begin{align*}
    (D\Phi_t)_{(g,\xi)} (g\cdot Y,\xi') 
    &= \frac{d}{ds}_{\vert_{s=0}} \left(\Phi_t(g e^{sY}, \xi+ s\xi')\right) 
     = 
     \frac{d}{ds}_{\vert_{s=0}} (g e^{sY}e^{t(\tilde\xi + s\tilde\xi')}, \xi + s\xi')\\ 
     &= (ge^{t\tilde\xi}\cdot (\mathrm{Ad}_{e^{-t\tilde\xi}}Y+ t\tilde \xi'), \xi').
\end{align*}
    Then 
    $$
    (\Phi_t^* \Theta)_{(g,\xi)} (g\cdot Y,\xi')=  \mathrm{Ad}_{e^{-t\tilde \xi}} Y + t\tilde \xi',  
    $$
    from which the result follows immediately.
\end{proof}

\begin{remark}[Expansions of the Maurer-Cartan form]\label{rem:Maurer-Cartan-form} We view the Maurer-Cartan form $\Theta_g$ at $g\in G$ as an element of $$\hom(T_gG,T_eG)=\hom(dL_g(e)(\gg),\gg)=(dL_g(e)(\gg))^*\otimes \gg=(dL_g(e)^*)^{-1}(\gg^*)\otimes\gg$$
with the dual map $dL_g(e)^*:T_g^*G\to \gg^*$ of the derivative $dL_g(e):\gg=T_eG\to T_gG$ of the left translation $L_g:G\to G$ by $g$ at $e$. Note that $G$ acts on $T^*G$ by 
$$\forall g,h\in G, \nu\in T_h^*G: \quad g\cdot \nu= (dL_{g^{-1}}(gh))^* \nu\in T^*_{gh}G.
$$ 
This means the left $G$-invariant $\gg$-valued form $g\mapsto \Theta_g$ can be viewed as an element $\Omega^1(G)^G\otimes \gg$, where  $\Omega^1(G)^G$ is the space of left $G$-invariant $1$-forms on $G$. If now $(Y_j)_j$ is a basis for $\gg$, then we have an expansion 
$$\Theta=\sum_j \eta_j Y_j := \sum_j \eta_j \otimes Y_j$$
with $\eta_j\in \Omega^1(G)^G$ such that the $(\eta_j(e))_j$ form the dual basis of $(Y_j)_j$ as $\Theta_e =\mathrm{id}_\gg$.
\end{remark}

We apply Remark~\ref{rem:Maurer-Cartan-form} to the Cartan and Bruhat decompositions of $\gg$,
$$
\gg = \kk \oplus \ss = \bar\nn \oplus\mm\oplus\aa\oplus\nn,
$$
from Subsections~\ref{subsec:basic definitions} and \ref{sectionevolutionofpol}.
Using the Bruhat decomposition, the Maurer-Cartan form can be written as
\begin{equation}\label{thetabruhat}
\Theta = \sum_{\alpha\in \Sigma^+} (\eta_\alpha E_\alpha + \eta_{-\alpha} E_{-\alpha}) + \sum_{j=1}^r \eta^A_j A_j + \sum_{i=1}^m \eta^M_i M_i,
\end{equation}
where the coefficients $\eta_\alpha, \eta_{-\alpha}$, $\eta^A_j, \eta^M_j$ are left-invariant one-forms on $G$.
Using the Cartan decomposition, on the other hand, we have
\begin{equation}\label{thetacartan}
\Theta = \sum_{\alpha\in \Sigma^+} \eta^K_\alpha E_\alpha^K+ \sum_{i=1}^m  \eta^M_i M_i + 
\sum_{\alpha\in \Sigma^+} \eta^{\ss}_\alpha E_\alpha^{\ss} + \sum_{j=1}^r \eta^A_j A_j,
\end{equation}
where $\eta^K_\alpha, \eta^{\ss}_\alpha$ are left-invariant one-forms on $G$. 
We have the equalities 
$$
\eta^K_\alpha = \frac12 (\eta_\alpha + \eta_{-\alpha}),\quad
\eta^\ss_\alpha = \frac12 (\eta_\alpha - \eta_{-\alpha}),\quad
\eta_\alpha = \eta^K_\alpha + \eta^{\ss}_\alpha,\quad 
\eta_{-\alpha} = \eta^K_\alpha - \eta^{\ss}_\alpha
$$
for $\alpha\in \Sigma^+$.

We will henceforth use the same notation for the pull-back to $G\times \aa_+^*$ under $p_G$ of the forms $\eta_\alpha, \eta_{-\alpha}$, $\eta^A_j, \eta^M_i, \eta^K_\alpha, \eta^{\ss}_\alpha.$

Let $p_X:G\to X=G/K$ be the canonical projection.

\begin{lemma}\label{lem:base-volume-form}
    We have on $G\times \aa_+^*$, for a nonzero constant $c_0$, 
    \begin{equation}\label{pull-backofdx}
    (p_X \circ p_G)^* dx =  c_0 \bigwedge_{\alpha\in \Sigma^+} \eta^{\ss}_\alpha \wedge \bigwedge_{j=1}^r \eta^A_j.
     \end{equation}
\end{lemma}

\begin{remark}
    The lemma also holds on $G$ if in (\ref{pull-backofdx}) we take $p_X^*dx$ on the left-hand side and on the 
    right-hand side the corresponding forms on $G$.
\end{remark}

\begin{proof}
   Note that the form $p_X^* dx$ is a basic form (see \cite{karshon:2016}) with respect to the right $K$-action on
$G$, that is, it is $K$-invariant and it is annihilated by vectors tangent to the $K$-orbits. Now,  under the $K$-action,  the left $G$-invariant form on the right-hand side of (\ref{pull-backofdx}), transforms by a factor given by  $\det \left(\mathrm{Ad}_K:\ss\to \ss\right)$ which is equal to 1, since $K$ is connected. (Note that $K\ni x \mapsto \det \left(\mathrm{Ad}_x:\ss\to \ss\right) \in \mathbb{R^\times}$ gives a group homomorphism $K\to \mathbb{R}^\times$ whose image is a compact connected subgroup of $\mathbb{R}^\times$.) Therefore, the right-hand side of (\ref{pull-backofdx}) is $K$-invariant. 

In addition, if $\mathcal{O}_e$ denotes the $K$-orbit through $e\in G$, for the right $K$-action on $G$, then $T_e{\mathcal{O}_e}=\kk\subset \gg$ and vectors tangent to ${\mathcal O}_e$ at $e$ annihilate the right-hand side of (\ref{pull-backofdx}). 
Since this form  is left $G$-invariant and the left action of $G$ commutes with the right $K$-action we obtain that the right-hand side of (\ref{pull-backofdx}) is annihilated by tangent vectors to the $K$-orbits at every point. 

It follows (see, for example, Theorem 1.1 of \cite{karshon:2016}) that there exists a unique top form on the symmetric space $X$ of 
which the right-hand side of (\ref{pull-backofdx}) is the pull-back under the  projection $p_X$. 
Since $p_X$ commutes with the left $G$-action we
conclude that this top form is also left $G$-invariant
and therefore corresponds, up to scale, to the $G$-invariant measure on $X$, $dx$.  
\end{proof}

Let $p_B:G\to B=G/MAN\cong K/M$ be the canonical projection. The $K$-invariant $\dim(B)$-form $db$ on $B$, which is associated with the measure $db$ defined at the end of Section \ref{Notation}, is not left $G$-invariant but, from Proposition I.5.1 in \cite{helgason:1984}, we have that 
$$
e^{2\rho(A(x,b))} p_B^* db
$$
is left $G$-invariant, where $\rho$ is the Weyl vector given by the half-sum of the restricted roots (counted with multiplicities) and $A(x,b)$ is defined in (\ref{defAxb}).

\begin{lemma}\label{lemmadb}
    We have, for a nonzero constant $c_1$,
    $$
     (p_B\circ p_G)^* db = c_1 e^{-2\rho(A(x,b))} \bigwedge_{\alpha\in \Sigma^+} \eta_{-\alpha}.
    $$
\end{lemma}

\begin{remark}
    The lemma also holds on $G$ if  we take $p_B^*db$ on the left-hand side and on the 
    right-hand side the corresponding forms on $G$.
\end{remark}

\begin{proof}
    We have that $e^{2\rho(A(x,b))}p_B^* db$ is left $G$-invariant and at $e\in G$ it vanishes along $\mathrm{Ker}\, (Dp_B)_e= \mm + \aa + \nn$. It follows that $p_B^*db$ at $e\in G$ must be proportional to
    $$
    \bigwedge_{\alpha\in\Sigma^+} \eta_{-\alpha}.
    $$
    The result then follows by the transitive left action of $G$ on itself.
\end{proof}

Recall now that every equivalence class in $(T^*X)_\mathrm{reg}$ has a unique representative of the form 
$[g,\xi]$, where $g\cdot M \in G/M$ and $\xi\in \aa_+^*,$ and that the Hamiltonian flow on $(T^*X)_\mathrm{reg}$ is given by $\varphi_t^{X_h} ([g,\xi]) = [ge^{t\tilde \xi}, \xi]$, 
and let 
$$
\hat p: G\times \aa_+^* \to (T^*X)_\mathrm{reg} \cong G/M \times \aa_+^*
$$
be the canonical projection. 

\begin{lemma} \label{lem:evolved-volume-form}
We have
$$
\hat p^* d\eta_t = c_0 \Phi_t^* \left( \bigwedge_{\alpha\in \Sigma^+} \eta^{\ss}_\alpha \wedge \bigwedge_{j=1}^r \eta^A_j\right).
$$
\end{lemma}

\begin{proof}
We have,
$$
\pi \circ \hat p = p_X \circ p_G
$$
and 
$$
\hat p \circ  \Phi_t = \varphi^{X_h}_t \circ \hat p.
$$
Therefore, 
$$\hat p^* d\eta_t = (\pi \circ \varphi_t^{X_h} \circ \hat p)^* dx = (\pi \circ \hat p \circ \Phi_t)^* dx = (p_X \circ p_G \circ \Phi_t)^*dx.$$     
The result then follows from Lemma \ref{lem:base-volume-form}.
\end{proof}

\begin{lemma}\label{lem:lifted-vol-asympt}
We have, for the pointwise limit,
\begin{equation}\label{eqlem:lifted-vol-asympt}
\lim_{t\to +\infty}  t^{-r} e^{-2\rho(A(x,b))} e^{-2t\rho(\tilde \xi)}\hat p^* d\eta_t = c_0 c_1^{-1} (-2)^{-|\Sigma^+|}\left((p_B\circ p_G)^* db\right) \wedge d\mathrm{vol}_{\aa^*}.
\end{equation}   
\end{lemma}

\begin{proof}
We have
$$
\mathrm{Ad}_{e^{-t\tilde \xi}} E_\alpha = e^{-t\alpha(\tilde \xi)} E_\alpha,\quad \mathrm{Ad}_{e^{-t\tilde \xi}} E_{-\alpha} = e^{+t\alpha(\tilde \xi)} E_{-\alpha}. 
$$
From the expressions for the Maurer-Cartan form in (\ref{thetabruhat}) and (\ref{thetacartan}) we see that Lemma \ref{lem:MC-flow-pullback} 
gives, explicitly in terms of the one-form coefficients of $\Theta$, that 
$$
\Phi_t^* \eta_\alpha = e^{-t\alpha(\tilde \xi)} \eta_\alpha,\,\,\,\,  \Phi_t^* \eta_{-\alpha} = e^{t\alpha(\tilde \xi)} \eta_{-\alpha}
$$
and
$$
\Phi_t^* \eta^A_j = \eta^A_j + t \hat \eta^A_j.
$$
Using the fact that $\alpha(\tilde \xi)>0$ for all $\alpha\in \Sigma^+$, 
the result now follows from Lemmas \ref{lemmadb}, \ref{lem:evolved-volume-form} and from (\ref{lettherebedlambda}).
\end{proof}

The form $p_B^*db$ on $G$ is a basic form for the quotient $$p_1:G\to G/M,$$ since $B=G/MAN$. Therefore, 
$$
p_1^*\widehat db = p_B^* db,
$$
for a differential form $\widehat db$ on $G/M$. 

\begin{remark}\label{remarkvaluesofc01}The constants $c_0$ and $c_1$ depend on the normalizations of Haar measures and invariant forms used. We use the normalizations given in \cite[\S~II.3.1]{helgason:2008}. Then, combining standard methods from Lie theory and Riemannian geometry, one can derive the formula 
$$c_0c_1^{-1} 
= (2\pi)^{-\frac{r}{2}}2^{\dim(K/M)}.$$ 
The detailed  derivation of this formula is given in Appendix \ref{appvolumeforms}.
\end{remark}

We are now ready to state the main result in this section.

\begin{proposition}\label{prop:canonical-form-lift} Under the identification $(T^*X)_\mathrm{reg} \cong G/M \times \aa_+^*$,
we have, for the pointwise limit,
$$\lim_{t \to +\infty} \, \, t^{-r} \, e^{-2 \left(\rho(A(x,b))+t\rho(\tilde \xi) \right)} \,  d\eta_t   = 
(-1)^{|\Sigma^+|} \, \widehat db \wedge \, 
d\lambda \, . 
$$
\end{proposition}

\begin{proof}
We have $$\lim_{t \to +\infty} \, \, t^{-r} \, e^{-2 \left(\rho(A(x,b))+t\rho(\tilde \xi) \right)} \,  d\eta_t   = c_0 c_1^{-1} 
(-2)^{-|\Sigma^+|} (2\pi)^{r/2}\, \widehat db \wedge \, 
d\lambda \, . 
$$
    This follows from the fact that the form on $G\times \aa_+^*$ given by (\ref{eqlem:lifted-vol-asympt}) is basic with respect to the projection $\hat p: G\times \aa_+^*\to (T^*X)_\mathrm{reg}\cong G/M \times \aa_+^*.$ Note that 
    if $p_2:G/M \to G/MAN$ is the natural projection then $p_B = p_2 \circ p_1.$ The proposition then follows from Remark \ref{remarkvaluesofc01}.
\end{proof}

\section{Analytic results}

In the present section we collect a series of analytic results that will be used in the next section to give an interpretation of the 
Fourier-Helgason transform in geometric quantization.

\subsection{The Fourier-Helgason transform}
Recall that, as in Section \ref{Notation}, we use the convenient notation of \cite[\S~II.3.4]{helgason:2008}, for the Iwasawa decompositions, $G=KAN$ and $G=NAK$, see also \eqref{iwasawas}. 

The following functions are joint (generalized) eigenfunctions of $G$--invariant differential operators on $X=G/K$ (see \cite[Prop.~II.3.14]{helgason:2008})
\begin{equation}
\label{e_spectral}
e_{\lambda, b}(x) := e^{(i \lambda + \rho)(A(x, b))} \, ,
    \end{equation} 
where $\lambda\in \aa_\CC^*$ and $A(x,b)$ is defined in (\ref{defAxb}). For the Laplace-Beltrami operator on $X$ with respect to the invariant  metric \footnote{If $\kappa$ denotes the Killing form we have $\langle v,w \rangle = -\kappa(v, \sigma(w))$, where $\sigma$ is the Cartan involution. See, for instance, Section 1 in \cite{anker-ostellari}. Note that, for complex $\lambda$, the eigenvalue in (\ref{e_eig}) is obtained by taking the complex bilinear extension of $\langle \lambda, \lambda\rangle.$}
given by the inner product $\langle \cdot, \cdot \rangle$ on $\gg$, $\Delta$, we have
\begin{equation}
\label{e_eig}
- \Delta \, e_{\lambda, b}(x) =  \left(|\lambda|^2 + |\rho|^2  \right) \, e_{\lambda, b}(x).  \, 
    \end{equation}

   Let us also, for later convenience, introduce the function
\begin{equation}
\label{e_eigTB}
E_{\lambda, b}([g, \xi]) := \pi^* (e_{\lambda, b})[(g,\xi)]=  e^{(i \lambda + \rho)(A(x,b))} ,
    \end{equation}
where $x=g\cdot K.$

Note that  $\mathfrak{a}^*_+$ parametrizes a subset of  $\widehat G_{\rm sph}$ (the space of equivalence classes of  tempered unitary spherical representations of $G$) whose complement has Plancherel measure zero. The generalized eigenfunctions  (\ref{e_spectral})
are the building blocks of the spectral and isotypical decomposition, which takes the form
\begin{equation}   \label{e1'}
L^2(X, d x) \cong \int_{\mathfrak{a}^*_+}^\oplus  \, {\mathcal{H}_\lambda} \, \frac{d\lambda}{\vert c(\lambda) \vert^2},
\end{equation}
where $c$ is the Harish-Chandra $c$-function, 
with ${\mathcal{H}_\lambda}$ being the unitary principal series representation with parameter $\lambda$. On the level of distributional vectors, it is embedded in the regular representation on $X$ by the Poisson transform, defined on  $L^2(B, db)$ by
\begin{equation}
\label{e-pt}
\forall x\in X:\quad {\mathcal{P}_\lambda}(F)(x) := \int_B \,  e^{(i \lambda + \rho)(A(x, b))} \, F(b) db   \, .    
\end{equation}

\begin{definition}
\label{def_FT}
The Fourier--Helgason (FH) transform 
is the  map $f  \mapsto  \widetilde f$ defined by the formula 
$$\widetilde f (\lambda, b) := \int_{X} \,
e^{(-i \lambda + \rho)(A(x,b))}  \, f(x) \, dx \, $$
whenever the integral is defined.
\end{definition}
 
Recall the following,
\begin{theorem}[{Thm. 1.5 and Thm. 1.9, in \cite[\S III.2]{helgason:2008}}]
\label{th_FT}
The FH transform defines a unitary isomorphism $$\mathcal{F} \, : \, L^2(X, dx)  \longrightarrow  L^2 \left(\mathfrak{a}_+^* \times  B, \frac{d\lambda db}{|c(\lambda)|^2}  \right)$$
and its inverse is given by
\begin{equation}
\label{e_IFT}
f(x) = \int_{\mathfrak{a}_+^* \times B} \, 
e^{(i \lambda + \rho)(A(x,b))}  \, \widetilde f(\lambda, b) \, \frac{d\lambda db }{|c(\lambda)|^2}  \, 
\end{equation}  
almost everywhere if $f$ and $\widetilde f$ are integrable. 
\end{theorem}

In the following sections we will establish some technical results that will be used in proving Theorem \ref{ThQG}.

\subsection{$\bar N$-integrals}\label{subsection_barN-integrals}

Below, we will make use of the following identities (see  \cite[\S~II.3.5]{helgason:2008}), 
\begin{equation}
\label{e_32}
\forall x \in X, g \in G , b \in B \,:\quad A(g\cdot x, g\cdot b) = A(x,b) + A(gK, g\cdot b),    
\end{equation}
so that 
\begin{equation}
\label{e_32'}
A(g_0e^{tH}K, b) = A(g_0e^{tH}K, g_0\cdot(g_0^{-1}\cdot b)) = A(e^{tH}K, g_0^{-1}\cdot b) +  A(g_0K, b). 
\end{equation}

\begin{lemma}\label{Lem:weighted moments}
Let $\eta\in \aa_+^*$. Then for every integer $m\ge 0$ we have 
\begin{equation}\label{eq:weighted-moments}\int_{\bar N}(1+\|H(\bar n)\|)^m e^{-(\rho+\eta)(H(\bar n))}\, d\bar n <\infty.
\end{equation}
\end{lemma}

\begin{proof}
    Choose $0<\delta<1$. As $H(\bar n)$ is contained in the closed dual cone of $\aa_+^*$ (see \cite[Cor.~IV.6.6]{helgason:1984}), there exists a constant $c_\eta$ such that
    \begin{equation*}\label{eq:H bar n}
     \forall \bar n\in \bar N:\quad \eta(H(\bar n))\ge c_\eta \|H(\bar n)\|.
    \end{equation*}
    Note that for each $m\in \mathbb{N}_0$ we have a constant $C_{m,\delta}>0$ such that
    $$(1+\|H(\bar n)\|)^m e^{-(\rho+\eta)(H(\bar n))}\le C_{m,\delta} e^{-(\rho+\delta\eta)(H(\bar n))}.$$
    As $\delta\eta \in \aa_+^*$ the convergence of the integral in \eqref{eq:weighted-moments} follows from \cite[Thm.~IV.6.14]{helgason:1984}.
\end{proof}

\begin{lemma}\label{Lem:Iwasawa properness}
Fix a norm  on $\bar\nn$. For $\eta\in \aa^*_+$ there exist constants $a,b>0$ such that
$$\forall X\in \bar\nn:\quad \eta(H(\exp X))\ge a\log(1+\|X\|)-b.$$
\end{lemma}

\begin{proof}
Pick a locally injective finite-dimensional representation $\pi:G\to \mathrm{GL}(V)$
with a regular highest weight. Equip $V$ with a $K$-invariant norm and pick a highest weight vector $v$ of unit norm. If $\omega\in \aa^*$ is the highest weight, then for $g=kan$ we find
$$ \|\pi(g)v\|^2= \|\pi(k)\pi(a)\pi(n)v\|^2 = \|\pi(a)v\|^2=e^{2\omega(\log a)}= e^{2\omega(H(g))}.$$

By \cite[Prop.~3.4.10]{gangolli-varadarajan} the map $\bar n\mapsto 2\omega(H(\bar n))$ is a proper map $\bar N\to \mathbb R_{\ge 0}$. Hence also $\bar N\to \mathbb R_{\ge 1},\ \bar n\mapsto e^{2\omega(H(\bar n))}$ is proper. As $\bar{\mathfrak{n}}\to \bar N, \ X\mapsto \exp(X)$ is a diffeomorphism, this implies that the map $\bar \nn\to V,\ X \mapsto \|\pi(\exp X)v\|^2$ is proper. As it is also polynomial, $\|\pi(\exp(X)v)\|$ has a lower bound of the form $c(1+\|X\|)^d$ with suitable $c,d>0$.
We may choose $c<1$. Since $\eta$ is regular, we find a constant $s>0$ such that $\eta-s\omega\in \aa^*_+$ whence
$$e^{\eta(H(\exp X))}\ge e^{s\omega(H(\exp X))}=\|\pi(\exp X)v\|^s\ge c^s(1+\|X\|)^{ds}.$$
Taking the logarithm yields
$$\eta(H(\exp X))\ge s\log c+ ds\log(1+\|X\|),$$
which proves the claim for $a=ds$ and $b=s\log\frac{1}{c}$.
\end{proof}

\begin{lemma}\label{lem:exp weighted tail} Given $\eta\in \aa^*_+$, $m\in \mathbb{N}_0$, $R>0 $, $\eta\in \aa_+^*$ and $0<\delta<1$ there exist $\sigma,C>0$ such that 
$$\forall t>0:\quad \int_{X\in \bar \nn, \|X\|>e^{Rt}}
(1+\|H(\exp X)\|)^m e^{-(\rho+\eta)H(\exp X)}\ dX \le Ce^{-\sigma t}.$$
\end{lemma}

\begin{proof}
Rewrite the integrand as
$$\big[(1+\|H(\exp X)\|)^m e^{-(\rho+\delta\eta)H(\exp X)}\big]
e^{-((1-\delta)\eta)H(\exp X)}.$$
To estimate the second factor we observe that for $\|X\|>e^{Rt}$ we have $(1+\|X\|)> \|X\|>e^{Rt}$. Lemma~\ref{Lem:Iwasawa properness} gives 
$$\eta(H(\exp X))\ge a\log(1+\|X\|)-b\ge aRt-b,$$
so for $C_1=\exp((1-\delta)b)$ we get
$$e^{-((1-\delta)\eta)H(\exp X)}\le C_1 e^{-(1-\delta)aRt}.$$
As the first factor is integrable by Lemma~\ref{Lem:weighted moments}, this proves the claim.
\end{proof}

\subsection{Schwartz boundary asymptotics}\label{subsection_Schwartz boundary asymptotics}

We begin by recalling, for completeness,  a property of Schwartz functions on a polyhedral cone $C\subset \mathbb{R}^n$.

\begin{definition}[Schwartz spaces on open polyhedral cones]
   Let $C\subseteq \RR^n$ be an open polyhedral cone. The \emph{Schwartz space} $\mathcal{S}(C)$ consists of all functions $f\in C^\infty(C)$ such that for every pair of multiindices $\alpha,\beta\in \mathbb{N}_0^n$, the seminorm
   $$\|f\|_{\alpha,\beta}:=\sup_{x\in C}\big|x^\alpha\partial^\beta f(x)|$$
   is finite.
\end{definition}

\begin{lemma}[Smooth extension to the boundary]\label{lem:Schwartz space}
Each $f\in\mathcal{S}(C)$ extends smoothly to the boundary, i.e. there is a smooth function defined on an open neighborhood of the closure $\overline{C}$ of $C$ in $\RR^n$ which restricts to $f$ on $C$.
\end{lemma}

\begin{proof}
$f$ and all its partial derivatives extend uniquely in a continuous way to $\overline{C}$. In fact, the condition $\|f\|_{0,\beta}=\sup_{x\in C}\big|\partial^\beta f(x)|<\infty$ for all $\beta$ implies that every partial derivative $\partial^\beta f$ is bounded on $C$. Hence, because $C$ is a convex polyhedral cone, we have that for $p,q\in C$ 
$$
\vert \partial^\beta f (p) -\partial^\beta f(q) \vert \leq a \vert\vert p-q\vert\vert^\gamma, 
$$
for some positive ($\beta$-dependent) constants $a$ and $\gamma.$ It follows that given any Cauchy sequence $p_n\to p$ converging to a point $p\in \partial C$ we will have that $\partial^\beta f (p_k)$ is a Cauchy sequence in $\mathbb{R}$ hence convergent. The Whitney extension theorem, as formulated in  \cite[Thm.~2.3.6]{Ho03}, guarantees that $f$ has a smooth extension to $\mathbb{R}^n$ which, moreover, can also be taken to be a Schwartz function.

\end{proof}

As a consequence of Lemma~\ref{lem:Schwartz space} we see that each element of $\mathcal{S}(C)$ is the restriction of a Schwartz function on $\RR^n$.  

Consider now the space $\mathcal{S}(\aa_+^*)$ of Schwartz functions on the polyhedral cone $\aa_+^*$.
Given a parameter manifold $P$ we write $\mathcal{S}(\aa^*_+\times P)$ for the smooth functions on $\aa^*_+\times P$ such that, for every compact subset $K\subseteq P$ and every choice of multiindices $\alpha,\beta,\gamma\in \mathbb{N}_0^n$, the seminorm 
\begin{equation}\label{eq:Schwartz-def}
   p_{K,\alpha,\beta,\gamma}(u):=\sup_{(\lambda,p)\in \aa^*_+\times K}\vert \lambda^\alpha \partial^\beta_p\partial^\gamma_\lambda u(\lambda,p)| 
\end{equation}
is finite. We equip $\mathcal{S}(\aa^*_+\times P)$ with the topology induced by these seminorms.

\medskip\indent
We fix an $\epsilon \in \aa^*_\CC$ (later it will tend to zero) with $\eta:=\mathrm{Re}(i\epsilon)\in \aa^*_+$. Now for a fixed Schwartz function $F\in \mathcal{S}(\aa^*_+\times B)$ we set
$$\Phi_{\lambda,g}(b'):= e^{(i\nu+\rho)(A(gK,g\cdot b'))}F(\lambda, g\cdot b')J(g;b')$$
for $\lambda\in \aa^*_+$, $\nu:=\lambda+\epsilon$, $g\in G$, $b'\in B$ and where 
$$
J(g;b')= \frac{db}{db'}(b'), \,\, b'=:g^{-1}\cdot b.
$$

Recall the expression for the Jacobian $J({ g_0};b_o)$ in terms of the function $A(x,b)$ given by (see  (51) in  \cite[\S~II.3.5]{helgason:2008}):
$$\frac{d(g^{-1}\cdot b')}{db'}=e^{2\rho(A(gK,b'))}.$$
Applying this with $g^{-1}=g_0$ gives
\begin{equation}\label{eq:Jac}J(g_0;b_o)= e^{2\rho(A(g_0^{-1}K,b_o))}= e^{2\rho(A(g_0^{-1}))}=e^{-2\rho(H(g_0))}. \end{equation}

Moreover, we will need the following variation of  \cite[Lem.~II.3.5]{helgason:2008}.
\begin{lemma}\label{TheHelgasonvariations}
Recall $\epsilon\in \aa_\CC^*$ with $\langle \eta, \alpha\rangle=\mathrm{Re}\, (i\langle \epsilon, \alpha\rangle)>0,$ for $\alpha\in \Sigma^+$. Both functions
$$
\lambda \mapsto c(\lambda+\epsilon) \,\,\, \mathrm{and} \,\,\,
\lambda \mapsto c(\lambda+\epsilon)^{-1},\,\,\, \lambda\in \aa_+^*,
$$
have slow growth, that is, each of their derivatives is bounded by a polynomial.
\end{lemma}
\begin{proof}
    We follow the proof of \cite[Lem.~II.3.5]{helgason:2008}. Helgason's 
    variable $z$ now becomes $z= \langle i\lambda + i\epsilon, \alpha_0\rangle$ for $\alpha_0$ a positive root. Since $\mathrm{Re}\,\langle i\epsilon,\alpha_0\rangle >0$ we also obtain that 
    both  $c(\lambda+\epsilon)$ and $c(\lambda+\epsilon)^{-1}$ are obtained as products of one-variable functions of the form
    $$
    \frac{\Gamma(a+ix)}{\Gamma(b+ix)}, \,\, x\in \RR,
    $$
    for constant $a,b>0$. (See also the proof of  \cite[Prop.~IV.7.2]{helgason:1984}.) The proof then follows exactly as in \cite{helgason:2008}.
\end{proof}

For the next proposition we introduce the functions
$$g_t(\lambda;g,H):=e^{(-i\nu+\rho)(tH)}\int_B e^{(i\nu+\rho)A(ge^{tH}K,b)}F(\lambda,b)\ db,$$
$$h_t(\lambda;g,H):=\frac{g_t(\lambda;g,H)}{c(\nu)},$$
and, for $H\in \aa$,
$$b_{t,H}:\bar N\to B,\quad \bar n\mapsto 
e^{tH}\bar ne^{-tH}\cdot b_o.$$

\begin{proposition}\label{prop:schwartzconv}
Let $C_G\subseteq G$ and $C_H\subseteq \aa_+$ be compact. For every choice of multiindices $\alpha, \beta\in \mathbb{N}_0^n$, and differential operators $D_g$ and $D_H$,
$$\sup_{\lambda\in\aa^*_+,(g,H)\in C_G\times C_H}|\lambda^\alpha\partial_\lambda^\beta D_gD_H(h_t-h_\infty)(\lambda;g,H)|\underset{t\to +\infty}\longrightarrow 0,$$
where 
$$h_\infty(\lambda;g)= e^{(i(\lambda+\epsilon)-\rho)H(g)} F(\lambda,g\cdot b_o).$$ 
In particular, $h_t\underset{t\to +\infty}\longrightarrow h_\infty$ in $\mathcal{S}(\aa_+^*\times (G\times \aa_+))$.
\end{proposition}

\begin{proof}
Recall that $\nu=\lambda+\epsilon$ and $\eta=\mathrm{Re}(i\epsilon)\in \aa^*_+$. Using \eqref{e_32'}, writing $b:=g^{-1} \cdot \tilde b$, by  \cite[Thm.~I.5.20]{helgason:1984} and by the identities in the proof of  \cite[Thm.~II.3.16]{helgason:2008} 
we have
\begin{align*}
g_t(\lambda;g,H)
&= e^{(-i\nu+\rho)(tH)}\int_B e^{(i\nu+\rho)(A(e^{tH}K,b))}\Phi_{\lambda,g}(b) d b\\
&=\int_{\bar N}e^{(i\nu-\rho)(H(e^{tH}\bar n e^{-tH}))}e^{-(i\nu+\rho)(H(\bar n))}\Phi_{\lambda,g}(b_{t,H}(\bar n))\ d\bar n\\
&=\int_{\bar N}E_{t,\nu}(X,H)W_\nu(X)\Phi_{\lambda,g}(b_t(X,H))\ d\bar n\\
&=\int_{\bar \nn}E_{t,\nu}(X,H)W_\nu(X)\Phi_{\lambda,g}(b_t(X,H))\ dX,
\end{align*}
where we  set
\begin{align*}
   E_{t,\nu}(X,H)&:= e^{(i\nu-\rho)(A_t(X,H))} 
                  := e^{(i\nu-\rho)(H(\bar n_t))}
                  := e^{(i\nu-\rho)(H(e^{tH}\bar n e^{-tH}))},\\
   W_\nu(X)&:= e^{-(i\nu+\rho)(A_0(X))} 
            := e^{-(i\nu+\rho)(H(\exp X))}  
            := e^{-(i\nu+\rho)(H(\bar n))},\\
   b_t(X,H)&:=  b_{t,H}(\bar n) = n_t\cdot b_o,     
\end{align*}
and the Lebesgue measure $d X$ is chosen to correspond to $d\bar n$ under the exponential map $\exp:\bar\nn\to \bar N$
 which is a diffeomorphism.
At $\tilde b=b_o$, the identities in the beginning of this section and of Section \ref{subsection_barN-integrals} give $A(gK,g\cdot b_o)=H(g)$ and $J(g;b_o)=e^{-2\rho(H(g))}$, so that
$$\Phi_{\lambda,g}(b_o)=e^{(i\nu-\rho)H(g)}F(\lambda,g\cdot b_o)=h_\infty(\lambda;g).$$

In view of the identity $c(\nu)=\int_{\bar N}e^{-(i\nu+\rho)H(\bar n)}d\bar n$ (see \cite[Thm.~IV.6.14]{helgason:1984}) we now arrive at the exact error identity
\begin{align}\label{eq:exact error}
&h_t(\lambda;g,H)-h_\infty(\lambda;g)\\
&=\frac{1}{c(\nu)}\int_{\bar N} e^{-(i\nu+\rho)H(\bar n)}
\left( e^{(i\nu-\rho)(H(e^{tH}\bar n e^{-tH}))}\Phi_{\lambda,g}(b_{t,H}(\bar n))-\Phi_{\lambda,g}(b_o)\right) \ d\bar n\nonumber\\
&= \frac{1}{c(\nu)} \int_{\bar\nn} W_\nu(X)\big( E_{t,\nu}(X,H)\Phi_{\lambda,g}(b_{t}(X,H))-\Phi_{\lambda,g}(b_o)\big) \ d\bar X.\nonumber
\end{align}

Let $^{+}\bar\aa := \{Y\in \aa\mid \forall \zeta\in \aa^*_+: \zeta(Y)\ge 0\}$ be the closed dual cone of the Weyl chamber $\aa^*_+$. Then  \cite[Cor.~IV.6.6]{helgason:1984} implies 
\begin{equation}\label{eq:Iwasawa-cone-monot}
  A_t(X,H),\  A_0(X,H)-A_t(X,H)\in {}^+\bar\aa. 
\end{equation}
Since  $\mathrm{Re} (i\nu)=\eta\in \aa_+^*$ is strictly positive on $^{+}\bar\aa\setminus\{0\}$ there exists a constant $C_\eta>0$ such that 
\begin{equation}\label{eq:Iwasawa-proj-bound}
    \|A_t(X,H)\|\le C_\eta \|A_0(X,H)\|.
\end{equation}
Choose $0<\delta<1$ sufficiently small so that $\rho-\delta\eta$ is dominant. Then by \eqref{eq:Iwasawa-cone-monot}
$$(\eta-\rho)(A_t(X,H))\le (1-\delta)\eta (A_t(X,H))\le (1-\delta)\eta(A_0(X,H))$$
so that
\begin{equation}\label{eq:global-weight-bound}
    |W_\nu(X)E_{t,\nu}(X,H)|\le e^{-(\rho+\delta\eta)(A_0(X,H))}
    \quad\text{and}\quad 
    |W_\nu(X)|\le e^{-(\rho+\delta\eta)(A_0(X))}.
\end{equation}
These are the global weight estimates we will need below.

Next we turn to derivatives of the terms under the integral in \eqref{eq:exact error}. We start with derivatives in the $\aa$-variable $H$. Let $V_Z$ be the fundamental vector field on $B$ generated by $Z\in \aa$, i.e.
\(V_Z\psi(b)=\frac{d}{ds}\Big|_{s=0}\psi(e^{sZ}\cdot b).\)
Further consider the derivative
$$v_Z:B\to \aa,\quad kM\mapsto \frac{d}{ds}\Big|_{s=0}H(e^{sZ}k)=\mathrm{pr}_{\aa}(\Ad(k^{-1})Z)$$ 
of the Iwasawa projection. Then we define the differential operator acting on smooth functions on $B$,
\begin{equation}\label{eq:bound-diff-op}
    D_{Z,\nu}\psi:= V_z\psi+(i\nu-\rho)(v_z-Z)\psi.
\end{equation}
All its coefficients are smooth on $B$ and their dependence on $\lambda$ is polynomial. 

In view of  $b_t(X,H)= \bar n_t\cdot b_o$, since $B= G/MAN,$ we find $b_t(X,H+sZ)=e^{tsZ}\cdot b_t(X,H)$. 
Set $\bar n_t\cdot b_0=k_tM$, so that $\bar n_tN=k_te^{H(\bar n_t)}N$ and $b_t(X,H)=k_tM$. We have
\begin{align*}
    A_t(X,H+sZ)
       & =H(e^{tsZ}\bar n_te^{-tsZ)})
     =H(e^{tsZ}\bar n_t)-tsZ\\
    &=H(e^{tsZ}k_te^{H(\bar n_t)})-tsZ  
     = H(e^{tsZ}k_t)+H(\bar n_t)-tsZ  \\
     &= H(e^{tsZ}b_t(X,H))+A_t(X,H)-tsZ.
\end{align*}

Differentiating yields the following exact identity for the partial derivative $\partial_{H,Z}$ in the $H$-variable in direction $Z$. 
\begin{equation}\label{eq:H-der-identity}
    \partial_{H,Z}\big(E_{t,\nu}\psi(b_t)\big)= t E_{t,\nu}(D_{Z,\nu}\psi)(b_t).
\end{equation}
Iteration yields 
\begin{equation*}\label{eq:iterated-H-der}
\partial_{H,Z_q}\cdots\partial_{H,Z_1}\big(E_{t,\nu}\psi(b_t)\big)= t^q E_{t,\nu}(D_{Z_q,\nu}\cdots D_{Z_1,\nu}\psi)(b_t)
\end{equation*}
and, in particular,
\begin{equation}\label{eq:H-der-bound}
\big|\partial_{H,Z_q}\cdots\partial_{H,Z_1}\big(E_{t,\nu}\psi(b_t)\big)\big|\le C t^q (1+\|\lambda\|)^q|E_{t,\nu}|\cdot \|\psi\|_{C^q(B)}. 
\end{equation}

Derivatives in the spectral $\aa^*$-variable $\lambda$ have the following behaviour:
\begin{itemize}
    \item[i)] derivatives of $W_\nu$ introduce extra components of $A_0$,
    \item[ii)] derivatives of $E_{t,\nu}$ introduce extra components of $A_t$, which are controlled by $A_0$ via \eqref{eq:Iwasawa-proj-bound},
    \item[iii) ] derivatives of the operators in \eqref{eq:bound-diff-op} retain smooth bounded coefficients and polynomial dependence on $\lambda$.
\end{itemize}

Derivatives in the $G$-variable $g$ only act on $\Phi_{\lambda,g}$ and do not interact with $W_\nu$ and $E_{t,\nu}$.

As a consequence of these considerations we find the following estimate for the derivatives of \(W_\nu\big( E_{t,\nu}\Phi_{\lambda,g}(b_{t})-\Phi_{\lambda,g}(b_o)\big)\) for fixed multiindex $\alpha,$ and smooth differential operators $D_g, D_H$, in the variables $g,H$ respectively, with $D_H$ of order $q$, for every prescribed $L$ and suitable $m$.
\begin{align}\label{eq:global-der-bound}
    &\big| \partial_\lambda^\beta D_g D_H\big(W_\nu( E_{t,\nu}\Phi_{\lambda,g}(b_{t})-\Phi_{\lambda,g}(b_o))\big)\big|\\
    &\le 
    C_L (1+t)^q(1+\|\lambda\|)^{-L}(1+\|A_0\|^m )\,e^{-(\rho+\delta\eta)A_0}.\nonumber
\end{align}
It only remains to justify the arbitrary decay index $L$. To do that, observe that the finitely many terms involve only finitely many $g$- and $\lambda$-derivatives of $\Phi_{\lambda,g}$, their $C^q(B)$-norms, and polynomial factors in $\lambda$. The Schwartz hypothesis supplies a bound with an arbitrarily large decay exponent, which absorbs those factors.

Lemma~\ref{lem:exp weighted tail} implies that for $\zeta\in \aa_+^*$ the integral
$$ \int_{\bar\nn}(1+\|A_0(X)\|)^m e^{-(\rho+\zeta)A_0(X)} dX$$
exists.
Therefore Lemma~\ref{lem:exp weighted tail} implies that the right hand side of \eqref{eq:global-der-bound} is integrable. Taking $t$ in a bounded interval controls the $(1+t)^q$ factor and justifies the differentiation under the integral, locally uniformly in the parameters.

In order to estimate the integral we split up the domain of integration. Recall the Weyl chamber $\aa_+\subset \aa$ corresponding to $\aa_+^*$ under the identification of $\aa$ and $\aa^*$. For a fixed compact subset $C_H\subset \aa_+$ we set
$$\kappa_0:=\min\{\alpha(H)\mid H\in C_H, \alpha\in \Sigma^+\}>0$$
and assume that our norm on $\gg$ is chosen such that the restricted-root spaces are orthogonal. If $\Omega_t:=\{X\in\bar\nn\mid e^{\frac{1}{2}\kappa_0 t}\ge \|X\|\} $ we obtain 
\begin{equation}\label{eq:contract-bound}
 \forall X\in \Omega_t:\quad    \|e^{t\ad H}X\| \le e^{-\frac{1}{2}\kappa_0 t}
\end{equation}
uniformly for $H\in C_H$. For every constant coefficient $H$-derivative of order $j$ we have
\begin{equation}\label{eq:contract-der-bound}
 \forall X\in \Omega_t:\quad    \|\partial_H^j e^{t\ad H}X\| \le  C_j t^j \|e^{t\ad H}X\|.
\end{equation}

Let now $Y_t = \log e^{tH}e^Y e^{-tH},\, Y\in \bar\nn$, and let 
$$F_{\lambda,g}(Y_t):=e^{(i\nu-\rho)H(\exp Y_t)}\Phi_{\lambda,g}(\exp(Y_t)\cdot b_o)-\Phi_{\lambda,g}(b_o).$$
Then $F_{\lambda,g}(0)=0$, and this remains true after any $g$- or $\lambda$-derivative. We claim that on the unit ball in $\bar\nn$, every fixed-order derivative  of $F_{\lambda,g}(Y)$ in $H,\lambda,g$ is bounded by finitely many seminorms of $\Phi_{\lambda,g}$ times a polynomial in $t$ and $\lambda$. In order to show this, note first that
$$\left| e^{(i\nu-\rho)H(\exp Y_t)} \right|= e^{(\eta-\rho)H(\exp Y_t)}$$
is bounded independently of $\lambda$. For zero $H$-derivatives the claim follows from the mean-value theorem and $F_{\lambda,g}(0)=0$. For positive $H$-derivative order, one applies the chain rule: every resulting term contains at least one derivative of $e^{t\ad H}Y$, hence a factor $\|e^{t\ad H}Y\|$ by \eqref{eq:contract-der-bound}. Products of additional such factors remain bounded by $\|e^{t\ad H}Y\|$ since $\|e^{t\ad H}Y\|\le 1$. This proves the claim.

Combining the above claim with \eqref{eq:contract-bound}, and differentiating $W_\nu$ if necessary, gives
\begin{align}\label{eq:contract-region-der-bound}
    \forall X\in \Omega_t:\quad &
   \big| \partial_\lambda^\alpha D_g D_H\big(W_\nu(X)( E_{t,\nu}(X,H))\Phi_{\lambda,g}(b_{t})-\Phi_{\lambda,g}(b_o))\big)\big|\\
    &\le 
    C_L (1+t)^q(1+\|\lambda\|)^{-L}\,e^{-\frac{1}{2}\kappa_0 t}(1+\|A_0(X)\|^m )\,e^{-(\rho+\eta)A_0(X)}.\nonumber   
\end{align}
Upon integration Lemma~\ref{Lem:weighted moments} yields
\begin{align}\nonumber 
    (1+\|\lambda\|)^{L}\int_{\Omega_t}&
   \big| \partial_\lambda^\beta D_g D_H\big(W_\nu(X)( E_{t,\nu}(X,H))\Phi_{\lambda,g}(b_{t})-\Phi_{\lambda,g}(b_o))\big)\big|\,dX\\
    &\le 
    C (1+t)^q\,e^{-\frac{1}{2}\kappa_0 t}\label{eq:contract-region-int-bound}
\end{align}
uniformly in $(\lambda,g,H)\in  \aa_+^*\times C_G\times C_H $, where $C_G\subseteq G$ is compact, for some constant $C>0$.

Next we apply \eqref{eq:global-der-bound} on the complement $\Omega_t^\mathrm{c}$ of $\Omega_t$ in $\bar\nn$. In this case Lemma~\ref{lem:exp weighted tail}  implies
\begin{align}\nonumber 
    (1+\|\lambda\|)^{L}\int_{\Omega_t^\mathrm{c}}&
   \big| \partial_\lambda^\beta D_g D_H\big(W_\nu(X)( E_{t,\nu}(X,H))\Phi_{\lambda,g}(b_{t})-\Phi_{\lambda,g}(b_o))\big)\big|\,dX\\
    &\le 
    C (1+t)^q\,e^{-\sigma_1t}\label{eq:tail-int-bound}
\end{align}
uniformly in $(\lambda,g,H)\in  \aa_+^*\times C_G\times C_H $, for some $C, \sigma_1>0$.

By Lemma~\ref{TheHelgasonvariations}, for every multi-index $\gamma$,
\begin{equation}\label{eq:inv-c-funct-der-bound}
    \big|\partial_\lambda^\gamma c(\lambda+\epsilon)^{-1}\big|\le C_\gamma (1+|\lambda|)^{M_\gamma}.
\end{equation}
Apply the product rule to \eqref{eq:exact error}. Choose $L$ in \eqref{eq:contract-region-int-bound} and \eqref{eq:tail-int-bound} large enough to absorb \eqref{eq:inv-c-funct-der-bound} and the desired weight $\vert\lambda\vert^\alpha$ in the statement of the proposition.  For $\sigma:= \min\{\frac{1}{2}\kappa_0, \sigma_1\}$ We obtain 
$$|\lambda^\alpha \partial_\lambda^\beta D_g D_H(h_t(\lambda;g,h)-h_\infty(\lambda;g)|\le 
 C (1+t)^q\,e^{-\sigma t}$$
 uniformly in $(\lambda,g,H)\in  \aa_+^*\times C_G\times C_H $, for some $C,\sigma >0$. This concludes the proof of Proposition~\ref{prop:schwartzconv}.

\end{proof}

Our original approach to Theorem B was to use Helgason's Fatou theorem, \cite[Thm.~II.3.16]{helgason:2008}. It turned out to be not quite sufficient for our purposes, which eventually led us to Proposition~\ref{prop:schwartzconv}. We used the proof of \cite[Thm.~II.3.16]{helgason:2008} in the proof of Proposition~\ref{prop:schwartzconv}, which in turn yields 
a variation of Helgason's theorem. As this variation may be of independent interest,  we present it here.

\begin{theorem}[A variant of Helgason's Fatou theorem]\label{th_HeThm-neu}
Fix $\lambda\in \aa_+^*$ and set $\nu=\lambda+\epsilon$ with $\epsilon\in\aa_\CC^*$ and $\eta=\mathrm{Re}(i\epsilon)\in \aa_+^*$. 
Then, for any $F \in \mathcal{S}(\aa^*_+\times B)$,
\begin{equation}
\label{e_IFT''} \nonumber
\lim_{t \to \infty} \,  \frac{e^{(-i\nu + \rho)(tH)}}{c(\lambda)}\int_{B} \, 
e^{(i \nu + \rho)(A(x(g,H,t),b))}  \, F(\lambda,b) \, db  \, =
 e^{(i(\lambda+\epsilon) - \rho)(H(g))}\, F(\lambda, g\cdot  b_o) \, ,
\end{equation}  
where both sides are viewed as elements of $\mathcal{S}(\aa^*_+\times (G\times \aa_+))$ in the variables $(\lambda,g,H)$ and the convergence is in $\mathcal{S}(\aa^*_+\times(G\times\aa_+))$ as well. For fixed $(\epsilon,H)$ and  $g=e$ this reduces to the limit formula
\begin{equation}\label{e_IFT'}
\lim_{t \to \infty} \,  e^{(-i(\lambda+\epsilon) + \rho)(tH)}\int_{B} \, 
e^{(i (\lambda+\epsilon) + \rho)(A(\exp(tH) \, x_o,b))}  \, F(\lambda,b) \, db  \, =
c(\lambda) \, F(\lambda, b_o) \,. 
\end{equation}  
\end{theorem}

\begin{proof}
This is immediate from Proposition~\ref{prop:schwartzconv}.
\end{proof}

\begin{remark}
As stated, Theorem \ref{th_HeThm-neu} is not a generalization of  \cite[Thm.~II.3.16]{helgason:2008} since it requires more regularity of the boundary values. Actually, one can adapt Helgason's proof directly and show that \eqref{e_IFT'} is valid as a pointwise limit for $F\in C(B).$
\end{remark}

\subsection{Uniform stationary phase lemma}

Recall that the measure $d\lambda$ on $\aa_+^*$ is given by $(2\pi)^{-r/2}$ times the Lebesgue measure.

\begin{lemma}\label{lem:stat-phase-compact-supp} Let $P$ be a  manifold and, for $t\geq 0,$ let 
$a_t,a_\infty\in \mathcal{S}(\aa_+^*\times P)$. Suppose that for a compact subset $C_P$ of $P$ and for every $N\in \mathbb{N}$, multiindex $\alpha$ and differential operator $D_P$ we have 
$$\sup_{(\lambda,p)\in \aa_+^*\times C_P} (1+\|\lambda\|)^N\,\big|\partial^\alpha_\lambda D_P(a_t-a_\infty)(\lambda,p)\big|\underset{t\to\infty}{\longrightarrow}0.$$
Set 
$$
I_t (a) (p) := \int_{\aa_+^*} e^{-\frac{it}{2}\|\lambda-\lambda_H\|^2} a(\lambda,p) d\lambda
$$
for $\lambda_H\in \aa_+^*$ and $a \in C^\infty(\aa_+^*\times P)$ whenever the integral exists.
If the $a_t$ and $a_\infty$ all vanish for $\lambda$ outside a fixed compact subset $C$ of $\aa_+^*$, then 
$$
\lim_{t\to +\infty}  t^{r/2}I_t (a_t) (p) = e^{\frac{-i\pi r}{4}} a_\infty(\lambda_H,p)
$$
and the convergence is uniform for $p\in C_P$.

\end{lemma}

\begin{proof}From  \cite[Lem.~2.8]{varadarajan97} and  \cite[Thm.~3.13]{zworski:2012} we have, for $a\in C^\infty_c(\aa_+^*,P)$ 
$$
\vert I_t (a)(p) - t^{-r/2} e^{-i\frac{\pi r}{4}} \sum_{k=0}^N \frac{t^{-k}}{k!} (R^ka)(\lambda_H, p)\vert \leq 
\frac{C_{N+1}}{(N+1)!} \sup_{|\alpha|<2N+r+1} \| \partial^\alpha a\| t^{-(N+1)},
$$
where $R$ is a differential operator of second order with constant coefficients in the $\aa^*$-variable.
We then have
\begin{align*}
&\vert I_t(a_t)(p) -t^{-r/2} e^{-\frac{i\pi r}{4}} \sum_{k=0}^N \frac{t^{-k}}{k!} (R^ka_\infty)(\lambda_H,p)\vert \leq\\
&\leq \vert I_t(a_t)(p) - t^{-r/2} e^{-\frac{i\pi r}{4}} \sum_{k=0}^N \frac{t^{-k}}{k!} (R^ka_t)(\lambda_H,p)\vert \\
&\qquad + 
\vert t^{-r/2} e^{-\frac{i\pi r}{4}} \sum_{k=0}^N \frac{t^{-k}}{k!} (R^ka_t -R^k a_\infty)(\lambda_H,p)\vert \\
&\leq \frac{C_{N+1}}{(N+1)!} \sup_{|\alpha|<2N+r+1} \| \partial^\alpha a_t\| t^{-(N+1)} + \|a_t-a_\infty\|_{C^{2N}
_{\aa_+^*\times C_P}} t^{-r/2} \sum_{k=0}^N \frac{t^{-k}}{k!} 
\\
&\leq \frac{C_{N+1}}{(N+1)!} \left( \|a_\infty\|_{C^l_{\aa_+^*\times C_P}} + \|a_t-a_\infty\|_{C^l_{\aa_+^*\times C_P}}\right) t^{-(N+1)}\\
&\qquad+ \|a_t-a_\infty\|_{C^{2N}_{\aa_+^*\times C_P}}t^{-r/2} \sum_{k=0}^N \frac{t^{-k}}{k!},
\end{align*}
for a sufficiently high $l$. Multiplying by $t^{r/2}$ and taking $N$ large enough the result follows.
\end{proof}

The previous result can be generalized to functions of rapid decay (i.e., not necessarily vanishing outside of a fixed compact subset).

\begin{lemma}\label{lem:stat-phase-compact-schwartz} Let $P$ be a  manifold and, for $t\geq 0,$ let $a_t\in \mathcal{S}(\aa^*_+\times P)$ be such that  $a_t\to a_\infty \in \mathcal{S}(\aa^*_+\times P)$. 
Let $\lambda_H\in \aa_+^*$ and 
$$
I_t (a) (p) := \int_{\aa_+^*} e^{-\frac{it}{2}\|\lambda-\lambda_H\|^2} a(\lambda,p) d\lambda
$$   
for $a \in \mathcal{S}(\aa^*_+\times P)$.
Then,
$$
\lim_{t\to +\infty}  t^{r/2}I_t (a_t) (p) = e^{\frac{-i\pi r}{4}} a_\infty(\lambda_H,p).
$$
\end{lemma}

\begin{proof}
    Let $\chi$ be a smooth bump function with compact support contained in an open neighbourhood of $\lambda_H$ in $\aa_+^*$ such that $\chi =1$ in an open neighbourhood of $\lambda_H$.
From Lemma \ref{lem:stat-phase-compact-supp}, to prove the lemma it suffices to show that 
$$
\lim_{t\to +\infty} t^{r/2} \int_{\aa_+^*} e^{-\frac{it}{2}\|\lambda-\lambda_H\|^2} (1-\chi) a_t(\lambda,p) d\lambda =0.
$$
Note that the integrand vanishes in a neighbourhood of the critical point $\lambda_H.$

From  \cite[Thm.2.7]{varadarajan97}, we have that if $a\in C^\infty(\aa_+^*\times P)$ has compact support which does not contain $\lambda_H$ then
$$
\lim_{t\to +\infty}
 \int_{\aa_+^*} e^{-\frac{i t}{2}\|\lambda-\lambda_H\|^2} a(\lambda;p) d\lambda = 0,
$$
uniformly in $p\in P$ and, even more precisely, in this case,
$$
\Big\vert\int_{\aa_+^*} e^{-\frac{i t}{2}\|\lambda-\lambda_H\|^2} a(\lambda;p) d\lambda\Big\vert \leq C_l 
\|a\|_{C^l(\aa_+^*\times P)} t^{-l},
$$
where  $l$ is an arbitrarily high order and $C_l$ a constant depending on $l$.

In order to apply this to the $a_t\in \mathcal{S}(\aa^*_+\times P)$ we use a partition of unity adapted to the faces of the closure of $\aa_+^*$. We can, moreover, choose a partition of unity whose elements have compact support. On interior charts the above argument yields the desired conclusion. Note that functions in $\mathcal{S}(\aa_+^*\times P)$, together with all their partial derivatives,  extend smoothly to $\bar \aa_+^* \times P$ by Lemma~\ref{lem:Schwartz space}. It also follows that the boundary values of the $a_t$ and their partial derivatives converge to the corresponding boundary values of $a_\infty$ and of its derivatives. For an open neighborhood $U$ of the covering which intersects  a boundary face 
of codimension $k$, we use linear coordinates $(u,v)\in U\cong [0,\infty[^k\times \mathbb{R}^{r-k}$.
We have, for the directions normal to a codimension $k$ face, noting that $\lambda_H$ is fixed in the interior of $\aa_+^*$ and for $U$ sufficiently small, that 
$\vert u^j(\lambda)-u^j(\lambda_H) \vert$ is bounded from below for all  $j=1,\cdots k$ and $\lambda\in U$. We then have, on $U$,
$$
\prod_{j=1}^k 
\left(it^{-1}\frac{u^j(\lambda)-u^j(\lambda_H)}{(u^j(\lambda)-u^j(\lambda_H))^2} \partial_{u_j}\right) \cdot e^{-\frac{i t}{2} \vert\vert\lambda-\lambda_H\vert\vert^2} = e^{-\frac{it}{2} \vert\vert\lambda-\lambda_H\vert\vert^2}.
$$
Applying this formula inside the integral, with $\varphi_U$ being the element of the partition of unity, gives for the contribution of $U$, 
$$
\int_{[0,+\infty[^k  \times\mathbb{R}^{n-k}} \prod_{j=1}^k 
\left(it^{-1}\frac{u^j(\lambda)-u^j(\lambda_H)}{(u^j(\lambda)-u^j(\lambda_H))^2} \partial_{u_j}\right) \cdot e^{-\frac{i t}{2} \vert\vert\lambda-\lambda_H\vert\vert^2} \varphi_U a_t\, dudv.
$$
Integration by parts for the integrals over the $u$ coordinates produces an overall factor of $t^{-k}$ times integrals over lowers dimensional faces and  along the $v$ coordinates. All integrands have absolute values whose $t$-dependence is controlled by the condition $a_t\to a_\infty$ in  
$\mathcal{S}(\aa_+^*\times P)$. The integrals over the $v$ coordinates, by applying stationary phase, contribute, at worst, a factor of 
$t^{-\frac12 (r-k)}.$ Hence, the contribution of the integral over $U$ is $O(t^{-\frac12 (r+k)})$ and, even after multiplication by 
$t^{\frac{r}{2}}$, vanishes in the limit $t\to +\infty.$

\end{proof}

\section{Lifting the geodesic flow to the quantum bundle: the Fourier-Helgason transform and geometric
quantization}

In the present section we will describe how the evolution of the spaces of polarized functions along the family of polarizations $\mathcal{P}_t, t\geq 0,$ can be described by a transform which we call quantum geodesic transform. 
This is an analogue, for Hamiltonian evolution in real time, of the Segal-Bargmann-Hall coherent state transform for Lie groups of compact type. The relation to the Fourier-Helgason transform is obtained in the limit $t\to +\infty.$

\subsection{The quantum bundle and the quantum geodesic transform}
We will now study the geometric quantization of $T^*X$ along the family of polarizations $\mathcal{P}_t, t\geq 0.$ The prequantum line bundle, in this case, is just the trivial bundle $L = T^*X \times \CC$, equipped with the standard Hermitian structure and prequantum connection 
$$
\nabla = d +i\theta, 
$$
where $\theta$ is the Liouville  (or tautological) one-form on $T^*X$, so that $d\theta = -\omega$ and the curvature of $\nabla$ is $-i\omega$. Sections of $L$ are identified with functions on $T^*X$. 
The Hilbert space of half-form corrected $\mathcal{P}_0$-polarized functions is then 
$$
\mathcal{H}_{\mathcal{P}_0}=\left\{ f\otimes \sqrt{dx}\,\, \vert \,\, f\in L^2(X,dx)\right\} \cong L^2(X,dx). 
$$
Using the description of the open dense subset of regular points $(T^*X)_\mathrm{reg}$ as $X \times   \aa_+^*\times B$ Lisiecki described the Hilbert space of $\mathcal{P}_F$-polarized functions (see  Theorem \ref{lisieckithm}) as
\begin{align*}
\mathcal{H}_{\mathcal{P}_F}&=\left\{ f e^{iS} \otimes \sqrt{d\lambda db}\,\,\vert \,\, f\in L^2\left(\aa_+^*\times B,\frac{d\lambda db}{\vert c(\lambda) \vert^2}\right)\right\} \\&\cong L^2\left(\aa_+^*\times B,\frac{d\lambda db}{\vert c(\lambda) \vert^2}\right).  
\end{align*}

Recall that the polarizations $\mathcal{P}_t, t>0,$ are obtained from the vertical polarization $\mathcal{P}_0$ by push-forward\footnote{We have $\mathcal{P}_t = (\varphi_{-t}^{X_h})_* \,  \Pm_0$ in terms of distributions on $T(T^*X)$ which corresponds to the fact that $\mathcal{P}_t$-polarized functions are obtained from $\mathcal{P}_0$-polarized functions by pull-back under $\varphi_t^{X_h}$, as described in Section \ref{sectionevolutionofpol}. } under the one-parameter group of diffeomorphisms $\varphi_t^{X_h}$ in (\ref{hamflow}), giving the Hamiltonian flow of $X_h$, where $h$ in (\ref{hamil}) is given by 
\begin{equation}
\label{e-ham_31}
h ([g,\xi])    = \frac{|\mu([g,\xi])|^2}{2} = \frac{|\xi|^2}{2} .
\end{equation}

The Hilbert space of $\mathcal{P}_t$-polarized functions, for $t\geq 0,$ can be conveniently described as follows. Each of the real polarizations $\mathcal{P}_t$ is an integrable distribution whose leaves define a simple foliation, $\mathcal{F}_{\mathcal{P}_t},$ of $(T^*X)_\mathrm{reg}.$ Recall that a \emph{simple foliation}, $\mathcal{F}$, on a smooth manifold $M$ is a regular foliation such that there is an atlas of foliated charts on $M$ such that each leaf of $\mathcal{F}$ intersects each chart of the atlas along at most one plaque. The quotient $M/\mathcal{F}$ is then a smooth manifold (which, in some cases, may not be Hausdorff) and the map $\pi_{\mathcal{F}}:M\to M/\mathcal{F}$ is a smooth submersion. Note that the algebra of smooth functions on $M/\mathcal{F}$ identifies naturally with the algebra of $\mathcal{F}$-invariant functions on $M$. (See, for example, Chapter 9 in \cite{ruiloja:2024}.) In our case, since $\mathcal{P}_t$ is obtained from the vertical polarization $\mathcal{P}_0$ by a diffeomorphism, we obtain the family of (Hausdorff) smooth manifolds
$$
X_t := (T^*X)_\mathrm{reg}/\mathcal{F}_{\mathcal{P}_t} \cong X,
$$
along with smooth submersions $\pi_{\mathcal{F}_{\mathcal{P}_t}}: (T^*X)_\mathrm{reg}\to X_t,$ where the algebra of smooth functions on $X_t$ is naturally identified with the algebra of $\mathcal{P}_t$-invariant functions. The form $\pi^* dx$ 
satisfies $L_Y (\pi^*dx) =0$ for any smooth section $Y$ of $\mathcal{P}_0$, where $L_Y$ denotes the Lie derivative along $Y$, and $\ker \pi^* dx \cong \ker D\pi_{{\mathcal F}_{\mathcal{P}_0}}$. Therefore, 
$\pi^* dx$ descends to a volume-form on $X_0\cong X$, which we also denote just by $dx$.  Likewise, since $\mathcal{P}_t$ is obtained from the vertical polarization by push-forward under a diffeomorphism, the form $d\eta_t$ in (\ref{detat}) defines a volume-form on $X_t$, still denoted just by $d\eta_t$. 

The evolution of polarized sections of the prequantum bundle along families of polarizations generated by Hamiltonian flow is described as follows. Consider the (half-form corrected) \emph{Kostant-Souriau prequantum operator} for $h$, ${\rm prQ}(h)$, and the operator $Q(h)$ which we call the \emph{quantum operator} for $h$:
\begin{eqnarray}
\label{e-2ops}
\nonumber {\rm prQ}(h) &:=& \left(-i \nabla_{X_h}  + h\right) \otimes 1 + 1 \otimes (-i L_{X_h})\\
&=&
\nonumber\left(-i X_h - h\right) \otimes 1 + 1 \otimes (-i L_{X_h})  \\
Q(h) &:=&  \frac{1}{2} \left( - \Delta - |\rho|^2  \right) \otimes 1,  \,  
\end{eqnarray}
where $L_{X_h}$ denotes the Lie derivative along $X_h$. (The quantum operator will be used below in Lemma \ref{lemmaquantumop}.)

It is a general feature holding for quantizations along families of real polarizations obtained by push-forward under Hamiltonian flow, that exponentiation of the prequantum operator maps polarized sections of the prequantum bundle to polarized sections.  More precisely, 
let $\psi = f\otimes \sqrt{dx}\in \mathcal{H}_{\mathcal{P}_0}$. Then, it is well-known that (see Chapter 9 in \cite{woodhouse:1991})
$$
e^{it {\rm prQ}(h)} \cdot \psi
$$
is $\mathcal{P}_t$-polarized, which is straightforward to verify explicitly.
Since the geodesic flow is analytic in $t$ (see (\ref{e_15})), on real-analytic functions on $T^*X$ the operator 
$e^{tX_h}$ acts simply as the pull-back $(\varphi_t^{X_h})^*.$  For $f\in C(X)$ we have therefore that (if the reader prefers, this can be taken as a definition)
\begin{align*}
    e^{it {\rm prQ}(h)} \cdot \left(f\otimes \sqrt{dx}\right)&= e^{-ith} e^{tX_h}\otimes e^{t L_{X_h}} \cdot \left(f\otimes \sqrt{dx}\right)\\ 
    &= e^{-ith} \left((\varphi_t^{X_h})^* \pi^*f \right)\otimes \sqrt{d\eta_t}
\end{align*}
is a $\mathcal{P}_t$-polarized half-form corrected section of the prequantum bundle. Note that 
the function $\left((\varphi_t^{X_h})^* \pi^* f\right)$, for $f\in C(X)$, is a $\mathcal{P}_t$-invariant function on $(T^*X)_\mathrm{reg}$ or, equivalently, it is a function on $X_t$. Since the phase factor $\exp{(-ith)}$ is universal, or  independent of $f$, we arrive at the following natural definition.

The Hilbert space of $\mathcal{P}_t$-polarized half-form corrected functions is defined by 
\begin{equation}\label{hilbertpt}
\mathcal{H}_{\mathcal{P}_t} = \overline{\left\{ f e^{-iht}\otimes \sqrt{d\eta_t}\,\, \vert \,\,f\in C^\infty (X_t),\,\, \,\,  \int_{X_t} \vert f\vert^2 {d\eta_t} <\infty\right\}},
\end{equation}
where the bar denotes norm completion. (Note that we are denoting the pull-back to $(T^*X)_{\mathrm{reg}}$ under 
$\pi_{\mathcal{F}_{{\mathcal P}_t}}$ of $f\in C^\infty (X_t)$  also by $f$.)

\begin{lemma}\label{lemmaprequantum}
    The operator
    $$
    e^{it {\rm prQ}(h)} : \mathcal{H}_{\mathcal{P}_0}\to \mathcal{H}_{\mathcal{P}_t}
    $$
is a unitary isomorphism of Hilbert spaces.
\end{lemma}

\begin{proof}

Notice now that 
$$
\vert\vert  e^{it {\rm prQ}(h)} \cdot \left(f\otimes \sqrt{dx}\right)\vert\vert_{\mathcal{H}_{\mathcal{P}_t}} = \vert\vert  
\left((\varphi_t^{X_h})^* \pi^*f \right) \otimes \sqrt{d\eta_t}\vert\vert_{\mathcal{H}_{\mathcal{P}_t}}
$$
where the Hermitian structure in $\mathcal{H}_{\mathcal{P}_t}$, in (\ref{hilbertpt}), is given by the integral of a form of top-degree on $X_t$ which is obtained from the form of top-degree on $X_0\cong X$, giving the usual $L^2$-norm of $f\otimes \sqrt{dx}$ in $L^2(X,dx)$, by pulling back under the diffeomorphism $\varphi_t^{X_h}$ and $\pi$. It follows, from the invariance of the integral under orientation-preserving diffeomorphisms,  that the operator 
$$e^{it {\rm prQ}(h)}: \mathcal{H}_{\mathcal{P}_0} \to \mathcal{H}_{\mathcal{P}_t}$$ is a unitary isomorphism.  
\end{proof}

\begin{lemma}\label{lemmaquantumop}
The operator 
$$
e^{-itQ(h)}: \mathcal{H}_{\mathcal{P}_0} \to \mathcal{H}_{\mathcal{P}_0}
$$
is a unitary isomorphism.
\end{lemma}

\begin{proof}
    We have that (for a $G$-invariant Riemannian structure) $X$ is a complete Riemannian manifold. It follows from Theorem 2.4 in \cite{Strichartz83} that the Laplace-Beltrami operator on $L^2(X,dx)$ is essentially 
    self-adjoint admitting, therefore, a unique self-adjoint extension. The lemma then follows from Stone's 
    theorem on one-parameter groups of unitary isomorphisms on Hilbert spaces (see, for instance, Chapter 10 in \cite{Hall:2013}). 
\end{proof}

The collection of Hilbert spaces $\left\{ \mathcal{H}_{\mathcal{P}_t}\right\}_{t\geq 0}$ forms a bundle of Hilbert spaces over the half-line $[0,+\infty)$,
$$ \mathcal{H} \to [0,+\infty).$$
We call it the \emph{quantum bundle}. Note that $\mathcal{H}$ has a natural smooth structure defined by declaring that sections of the form $t\mapsto e^{-iht} (\varphi_t^{X_h})^*\pi^* f \otimes \sqrt{d\eta_t}, f\in C^\infty(X)\cap L^2(X,dx),$ are smooth.

Lifting the geodesic flow to the quantum bundle means finding natural $G$-equivariant maps between the Hilbert spaces $\mathcal{H}_{\mathcal{P}_{t_1}}$ and $\mathcal{H}_{\mathcal{P}_{t_2}}$, for different values $t_1\neq t_2$, thus giving a natural identification between the quantizations of $T^*X$ in the polarizations $\mathcal{P}_{t_1}$ and $\mathcal{P}_{t_2}.$ 
To lift the geodesic flow to the quantum bundle, in analogy with several examples which involve Hamiltonian flows in complex time (see, for instance, \cite{baier.hilgert.kaya.mourao.nunes:2023, baier.ferreira.hilgert.mourao.nunes:2024}) we now define the quantum geodesic transform. 
This is an analogue, for real-time Hamiltonian evolution, of the Segal-Bargmann-Hall and generalized coherent state transforms, and may be interpreted as the parallel transport of a connection in the quantum bundle, as in \cite{baier.ferreira.hilgert.mourao.nunes:2024}, Section 4.2.

Note that in the description above, we may just as well take $t\in \mathbb{R}$ and obtain the quantum bundle
$$
\mathcal{H} \to \mathbb{R}.
$$

\begin{definition}
\label{defQG}
The \emph{quantum geodesic transform} (QGT) is the lifting of the geodesic flow on $T^*X$
to the Hilbert bundle given by the 
following family of operators
$$(U_t : \HH_{\Pm_0} \to \HH_{\Pm_t})_{t\in\mathbb{R}},$$ where
\begin{align}
   \label{e_QGT}
U_t &:= e^{i t {\rm prQ}(h)} \circ e^{-it Q(h)}\nonumber\\
&= \left( \left(\varphi^{X_h}_t\right)^* \otimes
e^{t{L}_{{X_h}}}\right) \, \circ \, \left( e^{-ith} \otimes 1 \right)\circ \left(
e^{\frac{it}{2} {\left(\Delta + \vert\rho\vert^2 \right)}}\otimes 1\right). 
\end{align}  
\end{definition}

\begin{remark}
    The convention of signs in the exponents in (\ref{e_QGT}) differs from the one in \cite{baier.hilgert.kaya.mourao.nunes:2023, baier.ferreira.hilgert.mourao.nunes:2024} and corresponds also to the different choice of sign for the $\nabla_{X_h}$ term in the definition of the prequantum operator in (\ref{e-2ops}). This is the convention that agrees with the usual conventions for the Fourier-Helgason transform, as shown in Section \ref{toinftyandbeyond}.
\end{remark}

\begin{proposition}
    $U_t: \mathcal{H}_{\mathcal{P}_0}\to \mathcal{H}_{\mathcal{P}_t}$ is, for every $t \in \RR$, a $G$-equivariant unitary isomorphism. 
\end{proposition}

\begin{proof}
The symplectic structure, the Hamiltonian $h$, $X_h$, the operator $L_{X_h}$ and $\Delta$ are $G$-invariant. The unitarity of $U_t$ follows since, from Lemmas \ref{lemmaprequantum}, \ref{lemmaquantumop}, it is a composition of unitary operators.
\end{proof}

Note that the operators $U_t, t\in \mathbb{R},$ define a one-parameter family of $G$-equivariant parallel transport operators on the quantum bundle $\mathcal{H}$ given by 
$$
\mathcal{U}_{t_1,t_2}:= U_{t_2} \circ U_{t_1}^{-1}: \mathcal{H}_{\mathcal{P}_{t_1}} \to \mathcal{H}_{\mathcal{P}_{t_2}},
$$
so that 
$$
\mathcal{U}_{t_2,t_3} \circ \mathcal{U}_{t_1,t_2} = \mathcal{U}_{t_1,t_3}.  
$$
These parallel transport operators define, therefore, a $G$-invariant unitary (flat) connection on $\mathcal{H}\to \mathbb{R}.$

Using the inverse Fourier transform (\ref{e_IFT}), we obtain  from (\ref{e_QGT}), \eqref{e_eig} and \eqref{e_eigTB}  that the quantum geodesic transform can be written as
\begin{eqnarray*}
    && U_t\left(f \, \otimes \sqrt{dx} \right)([g,\lambda_H]) = \\ 
    && \int_{\mathfrak{a}_+^*\times B } \, 
\left(\varphi^{X_h}_t\right)^* \left( E_{\lambda, b} \right)([g,\lambda_H]) 
\, e^{-\frac{it}{2} \left(|\lambda|^2 + |\lambda_H|^2\right)} \, \widetilde f(\lambda, b) \, \frac{d\lambda db}{|c(\lambda)|^2}  \, \otimes \, \sqrt{d\eta_t} \, ,
\end{eqnarray*}
where   $d\eta_t$ is defined in (\ref{detat}),  $dx$ is the $G$-invariant measure on $X$ {and $\sqrt{dx}, \sqrt{d\eta_t}$ denote the half-forms associated to $dx, d\eta_t$, respectively.} Also, $\lambda_H\in \aa^*_+$ is defined by $\langle \lambda, \lambda_H\rangle = \lambda (H), H\in \aa_+,$ where $\langle\cdot,\cdot\rangle$ denotes the inner product on $\aa^*$ induced from the one on $\aa.$

\subsection{The limit $t\to +\infty$ and the Fourier-Helgason transform}
\label{toinftyandbeyond}

In this section we prove our second main result, which is a more precise version of Theorem B in the introduction. 

Let $\mathcal{S}(X)$ denote the space of Schwartz functions on $X$ as defined in \cite{EguchiOkamoto77}. Recall that the Fourier-Helgason transform is a topological isomorphism between $\mathcal{S}  (X)$ and the space of Weyl group invariant Schwartz functions $\mathcal{S}(\aa^*\times B)_W$ (recall Definition 3 and Theorem 4 in \cite{EguchiOkamoto77} and the definition in the beginning of Section \ref{subsection_Schwartz boundary asymptotics})  which is also an isometry with respect to the $L^2$ metric (see Theorem 5 in \cite{EguchiOkamoto77}). 

\begin{definition}\label{defregularizedU}
    For $\epsilon\in \aa^*_\CC$ such that 
    $\mathrm{Re}\, (i\langle \epsilon, \alpha\rangle)>0$ for $\alpha\in \Sigma^+$, define the \emph{regularized quantum geodesic transform}
\begin{eqnarray}
\label{ee-ut2}
&& U_t^\epsilon \, : \, 
\mathcal{S}(X)\otimes \sqrt{dx}  \longrightarrow C^\infty (X\times \aa_+^*\times B)\otimes\sqrt{d\eta_t}\nonumber \\
\nonumber && U_t^\epsilon\left(f \, \otimes \sqrt{dx} \right)([g,\lambda_H])  := \\ \nonumber & &  e^{-it\epsilon(H)} \int_{\mathfrak{a}_+^*\times B } \, 
\left(\varphi^{X_h}_t\right)^* \left( E_{\lambda+\epsilon, b} \right)([g,\lambda_H]) 
\, e^{-\frac{it}{2} \left(|\lambda|^2 + |\lambda_H|^2\right)} \, \widetilde f(\lambda, b) \, \frac{d\lambda db }{|c(\lambda)|^2}  \, \otimes \, \sqrt{d\eta_t} .
\end{eqnarray}
Note that
the integral on the right-hand side defines a $\mathcal{P}_t$-polarized function.

\end{definition}

\begin{definition}\label{defU}
Let $\mathrm{U}$ be the following map 
\begin{eqnarray}
\label{e33}
\mathrm{U} \, : \, \HH_{\Pm_0} & \longrightarrow & \HH_{\Pm_{F}}   \\  
\mathrm{U} (f \otimes \sqrt{dx})(x,\lambda, b)  &:=&       \frac{{c(\lambda)}}{|{c(\lambda)}|} \, \mathcal{F}(f)(\lambda , b) \,  e^{i\lambda(A(x,b))} \, \otimes \sqrt{\frac{d\lambda \, db}{|c(\lambda)|^2}}\,   \, . \nonumber
 \end{eqnarray}
\end{definition}
Note that, as recalled in Theorem \ref{lisieckithm}, from Theorem 4.2 and Section 6 (Paragraph 6.3) of \cite{Lisiecki87} it follows that the  phase function, $e^{i \lambda(A(x,b))}$,
in the right-hand side of (\ref{e33}), occurs as a factor in the expression for 
Fourier-polarized functions so that the right-hand side indeed belongs to $\HH_{\Pm_{F}}$. 
Note also that $\mathrm{U}$ is unitarily equivalent to the Fourier transform in the sense that they differ by a phase function.

Our main result in this section is the following: 

\begin{theorem}\label{ThQG}
    One has
       $$
    \lim_{\epsilon\to 0} \lim_{t\to \infty} U_t^\epsilon =  (-1)^{-\vert\Sigma^+\vert/2} e^{-i\frac{\pi r}{4}}   \, \mathrm{U} ,
    $$
    in the sense that, 
    for $f\otimes \sqrt{dx}\in \HH_{\Pm_0}$, with $f\in \mathcal{S} (X)$,
    $$
    \lim_{\epsilon\to 0} \lim_{t\to \infty} U_t^\epsilon (f\otimes \sqrt{dx})=  (-1)^{-\vert\Sigma^+\vert/2} e^{-i\frac{\pi r}{4}}   \,\mathrm{U} (f\otimes \sqrt{dx}),
    $$
    where the limit is taken in the compact-open topology in $C^0 (X\times  \aa_+^*\times B)\otimes C^0(\mathfrak{L})$ where 
    $\mathfrak{L}$ is the (trivial) square-root of the (trivializable) bundle of $G$-invariant $(\dim X)$-forms on $(T^*X)_\mathrm{reg}\cong G/M \times \aa_+^*.$
\end{theorem}

\begin{proof}

    From Definition \ref{defregularizedU}, equations (\ref{e_15}) and (\ref{e_eigTB})  we have 
    \begin{eqnarray}
\nonumber && U_t^\epsilon\left(f \, \otimes \sqrt{dx} \right)(\eta_t)  = \\
\nonumber &=& \, \int_{\mathfrak{a}_+^* \times  B} \, e^{-it\epsilon(H)}
e^{(i (\lambda+\epsilon) + \rho)(A(x_t,\tilde b))}  \, e^{-\frac{it}{2} (|\lambda|^2  + |\lambda_H|^2)} \, \widetilde f(\lambda, \tilde b) \, \frac{d\lambda d\tilde b}{|c(\lambda)|^2}  \, \otimes \, \sqrt{ d\eta_t} \\
\nonumber &=&    \, \int_{\mathfrak{a}_+^* \times  B} \, e^{(-i(\lambda+\epsilon) + \rho)(tH)} \,
e^{(i (\lambda+\epsilon) + \rho)(A(x_t,\tilde b))}  \,  e^{-\frac{it}{2} (|\lambda|^2  + |\lambda_H|^2 -2 \lambda (H))} \, \widetilde f(\lambda, \tilde b) \, \frac{d\lambda d\tilde b}{|c(\lambda)|^2}  \, \\ \nonumber
&& \qquad t^{\frac{r}{2}} \, e^{\rho(A(x,b))}\otimes \, \left( e^{-\left(\rho(A(x,b))+t\rho(H)\right)}t^{-\frac{r}{2}}\right) \, \sqrt{d\eta_t} 
\end{eqnarray}

Proposition \ref{prop:canonical-form-lift} then gives  
$$
\lim_{\epsilon\to 0}\lim_{t\to \infty} U_t^\epsilon\left(f \, \otimes \sqrt{dx} \right)(\eta_t)  = 
$$
$$
\lim_{\epsilon\to 0}\lim_{t\to \infty} t^{\frac{r}{2}} \, e^{\rho(A(x,b))}\int_{\mathfrak{a}_+^* \times  B} \, e^{(-i(\lambda+\epsilon) + \rho)(tH)} \,
e^{(i (\lambda+\epsilon) + \rho)(A(x_t,\tilde b))}  \,  e^{-\frac{it}{2} |\lambda -\lambda_H|^2}
$$
$$
 \, \widetilde f(\lambda, \tilde b) \, \frac{d\lambda d\tilde b}{|c(\lambda)|^2}  \otimes (-1)^{\vert\Sigma^+\vert/2}  \sqrt{d\lambda_H \,  db}.
$$

If $f\in S(X)$ then $\tilde f \in \mathcal{S}(\aa_+^* \times B)$ \cite{EguchiOkamoto77}.
The integral over $B$ produces a smooth rapidly decreasing function of $\lambda$, for each $t$, that is
$$
g_t(\lambda):=\int_B  e^{(-i(\lambda+\epsilon) + \rho)(tH)} \, e^{(i (\lambda+\epsilon) + \rho)(A(x_t,\tilde b))} \, \widetilde f(\lambda, \tilde b)  \vert c(\lambda)\vert^{-2} d\tilde b
\in \mathcal{S}(\aa_+^*).
$$

Let 
$$
h_t (\lambda)= \frac{g_t (\lambda)}{c(\lambda+\epsilon)}.
$$
and
$$
h_\infty (\lambda) =  e^{(i(\lambda+\epsilon) -\rho)(A(x,b))}  \tilde f (\lambda, b) \vert c(\lambda)\vert^{-2}.
$$
(Note that the factor of $c(\lambda+\epsilon)^{-1}$ is $t$-independent, has slow growth by Lemma \ref{TheHelgasonvariations} and does not affect the rapid decay of $h_t$.)
From Proposition \ref{prop:schwartzconv}, we have that, as $t\to +\infty$,
$h_t \to h_\infty$ in the Schwartz topology. Applying Lemma \ref{lem:stat-phase-compact-schwartz}, we obtain 

$$
\lim_{t\to +\infty}  U_t^\epsilon\left(f \, \otimes \sqrt{dx} \right)   = 
$$
$$
=(-1)^{\vert\Sigma_+\vert} e^{-\frac{i\pi r}{4}} 
 e^{\rho(A(x,b))} e^{(i(\lambda_H +\epsilon)-\rho)(H(g))}\frac{c(\lambda_H+\epsilon)\tilde f (\lambda_H,b)}
 {\vert c(\lambda_H)\vert^2} \otimes \sqrt{d\lambda_H \,  db}.
$$

Taking now $\epsilon\to 0$, and recalling that $A(x,b)=H(g)$, we obtain

\begin{eqnarray*}
&& \lim_{\epsilon \to 0}\lim_{t\to\infty} U_t^\epsilon\left(f \, \otimes \sqrt{dx} \right)(\eta_t) = \\ &=&
 (-1)^{\vert\Sigma_+\vert} e^{-i\frac{\pi r}{4}}  \widetilde f(\lambda_H, b) e^{i\lambda_H(A(x,b))}\frac{1}{\bar c(\lambda_H)} \otimes  \sqrt{d\lambda_H d b}.
\end{eqnarray*}

That is,
\begin{eqnarray*}
 && \lim_{\epsilon \to 0}\lim_{t\to \infty} U_t^\epsilon\left(f \, \otimes \sqrt{dx} \right) (x,\lambda,b) = \\ & = &  (-1)^{\vert\Sigma^+\vert/2} e^{-i\frac{\pi r}{4}} \frac{{c(\lambda)}}{|{c(\lambda)}|} \, \mathcal{F}(f)(\lambda , b)  e^{i\lambda(A(x,b))} \, \otimes \sqrt{\frac{d\lambda \,  db}{|c(\lambda)|^2}},
\end{eqnarray*}
where $\eta_t$ was defined in (\ref{e_15}), we have set $\lambda_H=\lambda$ and $(x,\lambda, b)$ are related by (\ref{thebigdiffeo}).
\end{proof}

\begin{remark}
    The image of $\mathrm{U}$, which is the Fourier transform up to a phase factor, is the Hilbert space $\mathcal{H}_{\mathcal{P}_F}$ of half-form corrected polarized sections for the Fourier polarization $\mathcal{P}_F$. 
    Note that the functions in the image of $\mathrm{U}$ have precisely the form given in Theorem \ref{lisieckithm} and, as recalled in Section \ref{subsfourier},  from Section 6, paragraph 6.3 in \cite{Lisiecki87}, there is an isomorphism 
$$
\mathcal{H}_{\mathcal{P}_F} \cong L^2\left(\aa^*_+\times B, \frac{d\lambda db}{\vert c(\lambda)\vert^2}\right).
$$   
\end{remark}

 \bigskip

\appendix

\section{Volume forms and normalizing constants}\label{appvolumeforms}
The purpose of this appendix is to calculate the constants $c_0$ and $c_1$ showing up in Lemmas~\ref{lem:base-volume-form} and \ref{lemmadb} as a consequence of the uniqueness of invariant measures up to scalars. As we mentioned in Section~\ref{Notation}, we follow the normalizations of invariant measures chosen in \cite[\S~II.3.1]{helgason:2008}. Our calculations are of a differential geometric nature, so we need to see what these normalizations mean in terms of volume forms. This will allow us to keep track of the constants in the arguments providing integral formulas via the transformation formula for Lebesgue measures. 
After fixing orientation, measures are represented by positive
volume forms, and
change-of-variables formulas 
use absolute Jacobians.

In order to describe suitable volume forms for $G$ and the subgroups $K$, $A$, $N$, $\bar{N}$, $M$, the homogeneous spaces $G/K, K/M$, and the subspaces $\aa,\ss$ of $\gg$, we introduce the inner product $B_\sigma(v,w)=-\kappa(v,\sigma w)$, where $\kappa$ denotes the Killing form of $\gg$. It defines a left invariant Riemannian metric on $G$.

Put $d=\dim\nn$, $m=\dim\mm$ and $r=\dim\aa$.
Fix ordered $B_\sigma$-orthonormal bases
$(E_\alpha)$, $(M_i)$ and $(A_j)$
of $\nn$, $\mm$ and $\aa$,
with $(E_\alpha)$ adapted to the restricted
root spaces, and put
$E_{-\alpha}=\sigma E_\alpha$
in the same order.
In vector and coframe expressions,
$\alpha$ includes a basis index within
each restricted root space;
all wedge products use this order.
Write $E_\alpha^K=E_\alpha+E_{-\alpha}$
and $E_\alpha^\ss=E_\alpha-E_{-\alpha}$.
Orient $\nn,\bar\nn,\mm,\aa$ by these bases
and $\gg,\kk,\ss$ by the ordered blocks
$$
\begin{array}{ll}
\gg:
  &(E_{-\alpha}),\ (M_i),\ (A_j),\ (E_\alpha)\\
\kk:
  &(M_i),\ (E_\alpha^K/\sqrt2)\\
\ss:
  &(E_\alpha^\ss/\sqrt2),\ (A_j).
\end{array}
$$
The Lie groups carry the left-translated
orientations on every component.
Orient $G/K$ by $\ss$ and $K/M$
by $(E_\alpha^K/\sqrt2)$ in $\kk/\mm$.
The orientation of $G/K$ agrees
with the root-first order
in Lemma~\ref{lem:base-volume-form}.

\begin{remark}[Volume forms of Riemannian manifolds]\label{rem:def-volume-form} 
   Given an orientable Riemannian manifold $\mathcal{M}$ the tangent bundle $T\mathcal{M}$ and the cotangent bundle $T^*\mathcal{M}$ are isomorphic Euclidean vector bundles (the real counterpart of Hermitian vector bundles). Given an oriented orthonormal local frame for $T\mathcal{M}$ the dual local frame, consisting of $1$-forms, is also orthonormal.  The exterior product of these $1$-forms is the local volume form on $\mathcal{M}$. All these local volume forms are independent of the choice of oriented orthonormal frames and agree on overlaps. Thus they define a global form $d\mathrm{vol}_\mathcal{M}$ on $\mathcal{M}$ which is called the volume form of the Riemannian manifold $\mathcal{M}$. (For standard references, see Chapter VI of \cite{boothby:1975} or Chapter I of \cite{BottandTu:1982}.)
\end{remark}

\begin{remark}[Volume forms on Lie groups]\label{rem:volume forms on Lie groups} 
   Let $G$ be a Lie group and $B:\gg\times \gg\to \RR$ a Euclidean inner product. By left translation this defines a Riemannian metric on $G$.  Consider the linear isomorphism
   $$\gg^*\to\gg,\quad \lambda\mapsto \lambda^\flat$$
   defined by
   $$\forall Y\in \gg:\quad \lambda(Y)=B(\lambda^\flat,Y).$$
   Then the dual metric $B^*:\gg^*\times \gg^*\to \RR$  for $B$ is defined by 
   $$B^*(\lambda_1,\lambda_2)=\lambda_1(\lambda_2^\flat)=B(\lambda_1^\flat,\lambda_2^\flat).$$
   Given an oriented $B$-orthonormal basis $(Y_j)_j$ for $\gg$, the dual basis $(\lambda_j)_j$ for $\gg^*$ is $B^*$-orthonormal. Further,
   left translating the $Y_j$ yields a left $G$-invariant orthonormal frame for $TG$, which we still denote by $(Y_j)_j$. The dual orthonormal frame $(\eta_j)_j$ then consists of the left $G$-invariant $1$-forms (cf. Remark~\ref{rem:Maurer-Cartan-form}) for which $(\eta_j(e))_j$ is the dual basis for $(Y_j)_j$. Thus the volume form on $G$ with respect to $B$ is given by
   $$d\mathrm{vol}_G=\bigwedge_j \eta_j.$$
\end{remark}

\begin{remark}[Decomposition of volume forms]\label{rem:volume-forms decomposition}
\begin{enumerate}
\item[]
    \item[(i)] Volume forms on closed subgroups $H$ of $G$ are obtained as exterior products of dual bases of oriented orthonormal bases for $\mathfrak{h}$. Given orthonormal bases for $B_\sigma$-orthogonal subalgebras $\mathfrak{h}_1,\ldots,\mathfrak{h}_r$, the ordered concatenation of the $r$ bases is an orthonormal basis for $\mathfrak{h}_1\oplus\ldots\oplus\mathfrak{h}_r$. Now suppose that $\mathfrak{h}_1\oplus\ldots\oplus\mathfrak{h}_r=\gg$ and that this concatenation is oriented. Then this decomposition of $\gg$ induces isometric embeddings $\mathfrak{h}_\ell^*\hookrightarrow\gg^*$, where the inner products on the $\mathfrak{h}_\ell^*$  and $\gg^*$ are obtained from the respective restrictions of $B_\sigma$ via duality. Let $(Y_j^\ell)_{j=1,\ldots,\dim\mathfrak{h}_\ell}$ be a oriented orthonormal basis for $\mathfrak{h}_\ell$ and $(\eta_j^\ell)_{j=1,\ldots,\dim\mathfrak{h}_\ell}$ its dual basis for $\mathfrak{h}_\ell^*$. Then $(\eta_j^\ell)_{j=1,\ldots,\dim\mathfrak{h}_\ell}$ is orthonormal.  If we assume that each $\mathfrak{h}_\ell$ is the Lie algebra of a closed subgroup $H_\ell$ of $G$ the volume form for $H_\ell$ is given by 
$$d\mathrm{vol}_{H_\ell}=\bigwedge_{j=1}^{\dim\mathfrak{h}_\ell}  \eta_j^{\ell,H_\ell}, $$
where $\eta_j^{\ell,H_\ell}$ is the $H_\ell$-invariant $1$-form on $H_\ell$ with $\eta_j^{\ell,H_\ell}(e)=\eta_j^\ell\in \mathfrak{h}_\ell^*$. This volume form extends to a $G$-invariant $\dim\mathfrak{h}_\ell$-form 
$$d\mathrm{vol}_{H_\ell}^G=\bigwedge_{j=1}^{\dim\mathfrak{h}_\ell}  \eta_j^{\ell,G}, $$
on $G$, where $\eta_j^{\ell,G}$ is the $G$-invariant $1$-form on $G$ with $\eta_j^{\ell,G}(e)=\eta_j^\ell\in \gg^*$ (using the isometric embedding $\mathfrak{h}_\ell^*\hookrightarrow\gg^*$). But then
$$d\mathrm{vol}_G=\bigwedge_{j_1=1}^{\dim\mathfrak{h}_1}  \eta_{j_1}^{1,G}\wedge\ldots\wedge \bigwedge_{j_r=1}^{\dim\mathfrak{h}_r}  \eta_{j_r}^{r,G}
=d\mathrm{vol}_{H_1}^G\wedge\ldots\wedge d\mathrm{vol}_{H_r}^G.$$
\item[(ii)]
Now suppose conversely that for each $H_\ell$ we have a left Haar measure $\mu_{H_\ell}$. This left Haar measure is given by a left $H_\ell$-invariant $\dim\mathfrak{h}_\ell$-form $d\mu_{H_\ell}$, which has to be a multiple of the volume form 
$d\mathrm{vol}_{H_\ell}$. We define the proportionality constant by 
$$c_{H_\ell}d\mathrm{vol}_{H_\ell}=d\mu_{H_\ell}.$$
By definition we have  $d\mu_{H_\ell}(e)\in \bigwedge^{\dim\mathfrak{h}_\ell} \mathfrak{h}_\ell^*$,
but the isometric embedding $\mathfrak{h}_\ell^*\hookrightarrow\gg^*$ induces embeddings on exterior products. Thus we may say that $d\mu_{H_\ell}(e)\in \bigwedge^{\dim\mathfrak{h}_\ell} \gg^*$, which in turn defines a left $G$-invariant $\dim\mathfrak{h}_\ell$-form $d\mu^G_{H_\ell}$ on $G$. But then
$$d\mu^G_{H_1}\wedge\ldots\wedge d\mu^G_{H_r}$$
is a left $G$-invariant $\dim\gg$-form on $G$, i.e. a multiple of the left $G$-invariant $\dim\gg$-form $d\mu_G$ on $G$ defining the Haar measure $\mu_G$ of $G$. Comparing the values of the various forms at $e$ we obtain 
$$d\mathrm{vol}_{H_\ell}^G(e)
=d\mathrm{vol}_{H_\ell}(e)
=\frac{1}{c_{H_\ell}}d\mu_{H_\ell}(e)
=\frac{1}{c_{H_\ell}}d\mu^G_{H_\ell}(e),$$
which implies 
$$c_{H_\ell}d\mathrm{vol}_{H_\ell}^G=d\mu^G_{H_\ell}$$
and therefore
$$\prod_{\ell=1}^r c_{H_\ell}d\mathrm{vol}_G =d\mu^G_{H_1}\wedge\ldots\wedge d\mu^G_{H_r}.
$$
\end{enumerate}
\end{remark}

It is known that (see e.g. \cite[Chap. VI]{helgason:1984} or \cite[Chap. VI]{knapp:1996}) $\gg=\bar\nn\oplus \mm\oplus \aa\oplus\nn$ is an orthogonal decomposition with respect to this inner product and the decomposition of $\nn$ into root spaces for different restricted roots is orthogonal as well.  
Note that, on the other hand, the Iwasawa decomposition $\gg=\kk\oplus\aa\oplus \nn$ is only a direct vector space decomposition for which $\kk$ and $\nn$ are not orthogonal.

\begin{remark}[Bruhat decomposition of the Haar measure]\label{rem:c_G}
Recall  the formula 
$$\int_G f(g)\, dg=
\int_{\bar N\times M\times A\times N} f(\bar nman) e^{2\rho(\log a)} \, d\bar n\, dm\, da\, dn$$
from \cite[Prop.~I.5.21]{helgason:1984}. 
\begin{enumerate}
\item[(i)] In order to identify the left $G$-invariant form underlying this Haar measure we rewrite the formula as 
$$\int_G f(g)\, d\mu_G(g)=
\int_N\int_A\int_M\int_{\bar N} f(\bar nman) e^{2\rho(\log a)}  d\mu_{\bar N}(\bar n) d\mu_M(m) d\mu_A(a) d\mu_N(n),$$
where the right-hand side is an iterated integral of forms. Now we use projections $\mathrm{pr}_{\bar N}, \mathrm{pr}_M, \mathrm{pr}_A, \mathrm{pr}_N$ to pull back the forms $d\mu_{\bar N}, d\mu_M, d\mu_A, d\mu_N$ to forms $d\tilde\mu_{\bar N}, d\tilde\mu_M, d\tilde\mu_A, d\tilde\mu_N$ on $\bar N\times M\times A\times N$. The exterior product of these forms gives the product measure (see, for instance, Chapter I of \cite{BottandTu:1982}), so we have
$$\int_G f(g)\, d\mu_G(g)=
\int_{\bar N M A N} f(g) e^{2\rho(\log \mathrm{pr}_A(g))}  d\tilde\mu_{\bar N}(g)\wedge d\tilde\mu_M(g)\wedge d\tilde\mu_A(g)\wedge d\tilde\mu_N(g).$$
Thus the left $G$-invariant $\dim \gg$-form $d\mu_G$ on $G$ satisfies
$$d\mu_G(e)=d\tilde\mu_{\bar N}(e)\wedge d\tilde\mu_M(e)\wedge d\tilde\mu_A(e)\wedge d\tilde\mu_N(e).$$  
As the single forms are pull-backs by the projections, the values at $e$ are the forms  $d\mu_{\bar N}(e), d\mu_M(e), d\mu_A(e), d\mu_N(e)$ on the respective subalgebras $\bar\nn,\mm,\aa,\nn$ extended by zero on the complementary factors. But since $\gg$ is the orthogonal sum of these subalgebras, we can use the embeddings $\bar\nn^*,\mm^*,\aa^*,\nn^*\hookrightarrow\gg^*$ and obtain that  
\begin{align*}
   &d\tilde\mu_{\bar N}(e) =d\mu_{\bar N}(e) =c_{\bar N}\,d\mathrm{vol}_{\bar N}(e),\\
   &d\tilde\mu_M(e) =d\mu_M(e) =c_M\, d\mathrm{vol}_M(e),\\
   &d\tilde\mu_A(e) =d\mu_A(e) =c_{A}\,d\mathrm{vol}_A(e),\\
   &d\tilde \mu_N(e) =d\mu_N(e) =c_N\, d\mathrm{vol}_N(e).
\end{align*}
On the other hand the considerations in Remark~\ref{rem:volume-forms decomposition} show that 
$$d\mathrm{vol}_G(e)=d\mathrm{vol}_{\bar N}(e)\wedge d\mathrm{vol}_M(e) \wedge d\mathrm{vol}_A(e)\wedge d\mathrm{vol}_N(e).$$
Writing 
$$d\mu_G=c_G\, d\mathrm{vol}_G $$
we now find 
\begin{equation}\label{eq: c_G}
    c_G=c_{\bar N}c_Mc_Ac_N.
\end{equation}
Note that the Cartan involution $\sigma$ acts isometrically on $G$ and interchanges $N$ and $\bar N$. Moreover the normalization of $N$ is derived from the normalization of $d\bar n$  by applying $\sigma$. This implies $c_N=c_{\bar N}$.
\item[(ii)] 
The above integral formula implies that the form 
$$d\mu_G=e^{2\rho(\log \mathrm{pr}_A)}  d\tilde\mu_{\bar N}\wedge d\tilde\mu_M\wedge d\tilde\mu_A\wedge d\tilde\mu_N$$
is left $G$-invariant hence equal to the left $G$-invariant extension of 
$$d\tilde\mu_{\bar N}(e)\wedge d\tilde\mu_M(e)\wedge d\tilde\mu_A(e)\wedge d\tilde\mu_N(e)
=d\mu_{\bar N}(e)\wedge d\mu_M(e)\wedge d\mu_A(e)\wedge d\mu_N(e).$$
Thus we have
$$d\mu_G=d\mu^G_{\bar N}\wedge d\mu^G_M\wedge d\mu^G_A\wedge d\mu^G_N.$$
\end{enumerate}
\end{remark}

\begin{remark}[The volume form of $G$]\label{rem:volume form of G}
For a closer analysis of the relation between volume forms and Haar measures we recall the left $G$-invariant $1$-forms coming from the Maurer-Cartan form via \eqref{thetabruhat} and \eqref{thetacartan}. 
\begin{enumerate}
    \item [(i)]
Note that in view of the orthogonality of the Bruhat decomposition $\gg=\bar\nn\oplus \mm\oplus \aa\oplus\nn$  in Section~\ref{sec:Evolution} we may assume that $E_\alpha, A_j, M_i$  together with the $\sigma E_\alpha=E_{-\alpha}$  form an orthonormal basis for $\gg$. Then the $\frac{1}{\sqrt 2}E_\alpha^K = \frac{1}{\sqrt 2}(E_\alpha+E_{-\alpha})$ form an orthonormal basis for the orthogonal complement $\mm^{\perp_{B_\sigma}}\cap\kk$ of $\mm$ in $\kk$. Similarly, the $\frac{1}{\sqrt 2}E_\alpha^\ss = \frac{1}{\sqrt 2}(E_\alpha-E_{-\alpha})$ form an orthonormal basis for the orthogonal complement $\aa^{\perp_{B_\sigma}}\cap \ss$ of $\aa$ in $\ss$. Thus the $\frac{1}{\sqrt 2}E_\alpha^K, M_i, \frac{1}{\sqrt 2}E_\alpha^\ss, A_j$ form another orthonormal basis for $\gg$, which is adapted to the Cartan decomposition $\gg=\kk+\ss$. In view of Remarks~\ref{rem:Maurer-Cartan-form} and \ref{rem:volume forms on Lie groups} as well as \eqref{thetabruhat} and \eqref{thetacartan} this shows
\begin{align*}
d\mathrm{vol}_G
&=(-1)^{d(d+m+r)+mr}\bigwedge_{\alpha\in \Sigma^+} \eta_\alpha
\wedge\bigwedge_{\alpha\in\Sigma^+}\eta_{-\alpha}
\wedge\bigwedge_{j=1}^r\eta_j^A\wedge\bigwedge_{i=1}^m\eta_i^M \\
&= (-1)^{dr}
\bigwedge_{\alpha\in \Sigma^+} \sqrt2\,\eta_\alpha^K
\wedge\bigwedge_{i=1}^m\eta_i^M
\wedge\bigwedge_{\alpha\in \Sigma^+} \sqrt2\,\eta_\alpha^\ss 
\wedge\bigwedge_{j=1}^r\eta_j^A\\
&= 
(-1)^{dr} 2^{d}\bigwedge_{\alpha\in \Sigma^+} \eta_\alpha^K
\wedge\bigwedge_{i=1}^m\eta_i^M
\wedge\bigwedge_{\alpha\in \Sigma^+} \eta_\alpha^\ss 
\wedge\bigwedge_{j=1}^r\eta_j^A.
\end{align*} 
In the notation of Remark~\ref{rem:volume-forms decomposition} we obtain
$$  d\mathrm{vol}^G_{\bar N}=\bigwedge_{\alpha\in \Sigma^+}\eta_{-\alpha},\quad
    d\mathrm{vol}^G_M=\bigwedge_{i=1}^m \eta_i^M,\quad
     d\mathrm{vol}^G_A=\bigwedge_{j=1}^r \eta_j^A,\quad
    d\mathrm{vol}^G_N=\bigwedge_{\alpha\in \Sigma^+}\eta_\alpha.
    $$
\item[(ii)] Note that under the assumptions made in (i), interpreting $\eta^A_j(e)\in \aa^*\subseteq \gg^*$ as a translation invariant form on $\aa$, we obtain $$d\mathrm{vol}_{\aa} = \bigwedge_{j=1}^r \eta_j^A(e).$$ 
\item[(iii)] Using the orthogonal decomposition $\kk=\mm+(\mm^{\perp_{B_\sigma}}\cap\kk)$ we obtain 
\begin{equation}\label{eq:dvolK}
    d\mathrm{vol}_K = 2^{\frac{1}{2}d}\bigwedge_{i=1}^m \eta_i^M|_K \wedge \bigwedge_{\alpha\in \Sigma^+}\eta_{\alpha}^K|_K
    =
   2^{\frac{1}{2}d}\bigwedge_{i=1}^m \iota^*\eta_i^M \wedge \bigwedge_{\alpha\in \Sigma^+}\iota^*\eta_{\alpha}^K,
\end{equation} 
where we denote the inclusion of $K$ into $G$ by $\iota$.  We introduce the corresponding proportionality factor $c_K$ by 
$$d\mu_K=c_K\, d\mathrm{vol}_{K}.$$
\item[(iv)]
As $\mu_K=dk$ and $\mu_M=dm$ are normalized to be probability measures (see Subsection~\ref{Notation}) we get 
$$c_K^{-1}=\mathrm{vol}(K):=\int_K d\mathrm{vol}_K$$ 
and 
$$c_M^{-1}=\mathrm{vol}(M):=\int_M d\mathrm{vol}_M$$
\item[(v)]
Recall that we cannot apply the arguments for the Bruhat decomposition to the Iwasawa decomposition as $\gg=\kk+\aa+\nn$ is not orthogonal, so there is no a priori reason why $c_G=c_K c_A c_N$ should hold, which in view of \eqref{eq: c_G}  would imply $c_K=c_M c_{\bar N}$.
\end{enumerate}
\end{remark}

\bigskip

\begin{remark}[Homogeneous spaces]
 \begin{enumerate}
 \item[]
     \item[(i)] 
   Note that the group theoretic way to construct invariant measures on the homogeneous spaces $G/K$ and $K/M$ is via the formulas 
\begin{equation}\label{eq:quotient measure GK}\int_G f(g)\, dg= \int_{G/K} \int_K f(gk)\, dk \, d\mu_{G/K}(gK)
\end{equation}
and
\begin{equation}\label{eq:quotient measure KM}\int_K f(k)\, dk= \int_{K/M} \int_M f(km)\, dm \, d\mu_{K/M}(kM), 
\end{equation}
see \cite[Thm.~I.1.9]{helgason:1984}. From this we get
$$\int_G f(g)\, d\mathrm{vol}_G(g)= \frac{c_K}{c_G}\int_{G/K} \int_K f(gk)\, d\mathrm{vol}_K(k) \, d\mu_{G/K}(gK)$$
and 
$$\int_K f(k)\, d\mathrm{vol}_K(k)= \frac{c_M}{c_K}\int_{K/M} \int_M f(km)\, d\mathrm{vol}_M(m) \, d\mu_{K/M}(kM).$$
\item[(ii)]
The fact that $K$ is compact with normalized Haar measure allows us to identify 
$d\mu_{G/K}$  as the push-forward of $dg$ with respect to the projection $p_X:G\to X=G/K$. In fact, applying \eqref{eq:quotient measure GK} to a sequence of right $K$-invariant smooth functions $f_n$ approximating the indicator function $I_{p_X^{-1}(E)}$ of a compact $K$-saturated set $p_X^{-1}(E)\subseteq G$ with $E\subseteq G/K$ compact yields
$$\mu_{G/K}(E)=\mu_G(p_X^{-1}(E))=   ((p_X)_*\mu_G)(E),$$
where as before $d\mu_G(g)=dg$. Similarly, $d\mu_{K/M}$  is the push-forward of $dk$ with respect to the projection $p_B\vert_K:K\to B=K/M$.
\item[(iii)]
The uniqueness statement in \cite[Thm.~I.1.9]{helgason:1984} says that there exist constants $c_{G/K}$ and $c_{K/M}$ such that 
$$d\mu_{G/K}=c_{G/K}\, d\mathrm{vol}_{G/K}\quad\text{and}\quad
d\mu_{K/M}=c_{K/M}\, d\mathrm{vol}_{K/M}$$
for the (invariant) volume forms of the homogeneous Riemannian manifolds $G/K$ and $K/M$.
\end{enumerate} 
\end{remark}

\begin{remark}[The volume forms of $G/K$ and $K/M$]\label{rem:volume forms X and B}
 \begin{enumerate}
 \item[]
     \item[(i)] As the $\frac{1}{\sqrt{2}}E_\alpha^\ss$ form an orthonormal basis for the orthogonal complement $\aa^{\perp_{B_\sigma}}\cap\ss$ of $\aa$ in $\ss$,  whose dual basis is formed by the $\sqrt{2}\,\eta_\alpha^\ss$, the volume form $d\mathrm{vol}_\ss$, viewed as a translation invariant $\dim\ss$-form on $\ss$, is given by 
     $$d\mathrm{vol}_\ss = (-1)^{dr} 2^{\frac{1}{2}d}
\bigwedge_{j=1}^r \eta_j^A(e) \wedge \bigwedge_{\alpha\in \Sigma^+}\eta_{\alpha}^\ss(e).$$
     \item[(ii)] The compact and connected group $K$ preserves $\ss^*$  and its metric. Therefore $d\mathrm{vol}_\ss$  is $K$-invariant. Identifying $T_{eK}(G/K)$ with $\ss$ we derive from it a $G$-invariant $\dim(G/K)$-form on $G/K$ which is nothing but the volume form $d\mathrm{vol}_{G/K}$. Thus we have
\begin{equation}\label{eq:dvolGK}d\mathrm{vol}_{G/K} = (-1)^{dr} 2^{\frac{1}{2}d}
\bigwedge_{j=1}^r \eta_j^A \wedge \bigwedge_{\alpha\in \Sigma^+}\eta_{\alpha}^\ss,
\end{equation}
where we identify basic right $K$-invariant forms on $G$ with forms on $G/K$.
\item[(iii)] By \eqref{eq:dvolK} the compact group $M$ preserves the subspace $(\mm^{\perp_{B_\sigma}}\cap \kk)^*$ of $\kk^*$ spanned by the $\eta_\alpha^K(e)$.
Its orientation is also preserved:
$K$ is connected and preserves
the orientation of $\ss$,
while $M$ fixes $\aa$ pointwise.
For regular $H\in\aa$,
$\ad H:\mm^{\perp_{B_\sigma}}\cap\kk
\to\aa^{\perp_{B_\sigma}}\cap\ss$
is an $M$-equivariant isomorphism.
Thus $M$ acts with determinant $1$
on $\kk/\mm$, even if $M$ is disconnected.
Therefore we have 
$$d\mathrm{vol}_{\mm^{\perp_{B_\sigma}}\cap \kk} = 2^{\frac{1}{2}d}
\bigwedge_{\alpha\in \Sigma^+}\eta_{\alpha}^K(e).
$$
and 
$$d\mathrm{vol}_{K/M} = 2^{\frac{1}{2}d}
\bigwedge_{\alpha\in \Sigma^+}\eta_{\alpha}^K|_K
=2^{\frac{1}{2}d}
\bigwedge_{\alpha\in \Sigma^+}\iota^*\eta_{\alpha}^K,
$$
noting that $\bigwedge_{\alpha\in \Sigma^+}\eta_{\alpha}^K$ is right $M$-invariant.
 \end{enumerate}   
\end{remark}

\begin{remark}[The constants $c_{G/K}$ and $c_{K/M}$]\label{rem:volGK} 
\begin{enumerate}
\item[]
    \item[(i)]
In order to relate $d\mathrm{vol}_{G/K}$ and $d\mu_{G/K}$ we consider the left-invariant form 
$$\beta:= 2^{\frac{1}{2}d}\bigwedge_{i=1}^m \eta_i^M \wedge \bigwedge_{\alpha\in \Sigma^+}\eta_{\alpha}^K$$
on $G$ and recall that \eqref{eq:dvolK}  implies $\iota^*(\beta)=d\mathrm{vol}_K$. Moreover, \eqref{eq:dvolGK} together with Remark~\ref{rem:volume form of G}(i) gives
\begin{align*}
    \beta\wedge p_X^* d\mathrm{vol}_{G/K}
&=(-1)^{dr} 2^{d}\bigwedge_{i=1}^m \eta_i^M \wedge \bigwedge_{\alpha\in \Sigma^+}\eta_{\alpha}^K \wedge 
\bigwedge_{j=1}^r \eta_j^A \wedge \bigwedge_{\alpha\in \Sigma^+}\eta_{\alpha}^\ss\\
&=(-1)^{d(m+r)}d\mathrm{vol}_G.
\end{align*}
Integration over the fibers of $p_X$ yields
$$\int_G f\,  d\mathrm{vol}_G 
=  \int_G f |\beta\wedge p_X^*\, d\mathrm{vol}_{G/K}| 
= \int_{G/K} I(f\beta) \, d\mathrm{vol}_{G/K}, 
$$
where $|\cdot|$ denotes the density associated with the top form, the fibers have their left-translated $K$-orientations, and
$$(I(f\beta))(gK)
:=\int_{p_X^{-1}(gK)} f\beta 
= \int_{gK} f \beta = \int_{K} f(gk)\, d\mathrm{vol}_K(k).  $$
Together we obtain
$$\int_G f\,  d\mathrm{vol}_G = \int_{G/K}\int_{K} f(gk)\, d\mathrm{vol}_K(k)\, d\mathrm{vol}_{G/K}(gK), $$
which is the volume form analogue of \eqref{eq:quotient measure GK}. Thus, we finally see that 
$$\frac{c_G}{c_K}\, d\mathrm{vol}_{G/K} = d\mu_{G/K} \quad\text{and}\quad c_{G/K}=\frac{c_G}{c_K}.$$
\item[(ii)]
Similarly, we find 
$$\frac{c_K}{c_M}\, d\mathrm{vol}_{K/M} = d\mu_{K/M} \quad\text{and}\quad c_{K/M}=\frac{c_K}{c_M}=\frac{\mathrm{vol}(M)}{\mathrm{vol}(K)}.$$
As $K$ is compact and $\mu_{K/M}$ is a probability measure we can conclude $c_{K/M}^{-1}=\mathrm{vol}(K/M)$, so that
$$\mathrm{vol}(K/M) = \frac{\mathrm{vol}(K)}{\mathrm{vol}(M)}.$$
\end{enumerate}
\end{remark}

\begin{remark}\label{rem:c0}
A comparison  of \eqref{pull-backofdx} and its proof with \eqref{eq:dvolGK} shows that
$$ c_{G/K}\,  d\mathrm{vol}_{G/K} = dx = c_0 2^{-\frac{1}{2}\dim(\nn)}d\mathrm{vol}_{G/K}$$
and, using Remark~\ref{rem:volGK},
$$c_0 = 2^{\frac{1}{2}\dim(\nn)}\, c_{G/K}= 2^{\frac{1}{2}\dim(\nn)}\,  \frac{c_G}{c_K}.$$

\end{remark}

Next we analyze the consequences of the chosen normalization of $d\bar n$ for the value of $c_N=c_{\bar N}$. To this end we have a closer look at the proof of \cite[Thm.~I.5.20]{helgason:1984}.

\begin{remark}[Calculating $c_N$]\label{rem:cN}
    Consider the map $\varphi:\bar N\to K/M,\ \bar n\mapsto k_1(\bar n) M$, see \eqref{iwasawas}. This is a diffeomorphism onto an open dense subset of $K/M$ whose complement has measure zero. The differential calculated below has positive determinant at $e$, $\varphi$ is orientation preserving. Therefore there  exists a smooth function $\tilde \psi$ on $\bar N$ such that
    $$\tilde \psi \,  d\mathrm{vol}_{\bar N} = \varphi^* \,  d\mathrm{vol}_{K/M}.$$
    
    Then for $F:K/M\to \CC$ we have
    \begin{align*}c_{K/M}^{-1}\int_{K/M} F\, d\mu_{K/M}
    &=\int_{K/M}  F\,d\mathrm{vol}_{K/M}=\int_{\bar N}  \varphi^*(F\,d\mathrm{vol}_{K/M})\\
    &=\int_{\bar N}  (\varphi^*F) \varphi^*\,d\mathrm{vol}_{K/M}
    =\int_{\bar N}  (F\circ\varphi) \varphi^*\,d\mathrm{vol}_{K/M}\\
    &=\int_{\bar N} (F\circ\varphi)  \tilde\psi \, d\mathrm{vol}_{\bar N}=c_{\bar N}^{-1}\int_{\bar N} (F\circ\varphi)  \tilde\psi \, d\mu_{\bar N}.
    \end{align*}
    \cite[Thm.~I.5.20]{helgason:1984} yields the equality
    $$\int_{K/M} F \, d\mu_{K/M} = \int_{\bar N} (F\circ\varphi)  e^{-2\rho\circ H} \, d\mu_{\bar N}.$$
    Together we get
    $$\int_{\bar N} (F\circ\varphi)  e^{-2\rho\circ H} \, d\mu_{\bar N}=\frac{c_{K/M}}{c_{\bar N}}\int_{\bar N} (F\circ\varphi)  \tilde\psi \, d\mu_{\bar N},$$
    so that 
    $$\tilde\psi(\bar n) = \frac{c_{\bar N}}{c_{K/M}}e^{-2\rho(H(\bar n))} = c_{\bar N}\mathrm{vol}(K/M)e^{-2\rho(H(\bar n))}$$
    and in particular
    $$c_{\bar N}=\frac{\tilde\psi(e)}{\mathrm{vol}(K/M)}.$$
    
In order to calculate $\tilde\psi(e)$ we need to determine $\varphi^*\,  d\mathrm{vol}_{K/M}(e)$, so we want to calculate the derivative $d\varphi(e):\bar \nn\to \kk/\mm$. The derivative of the Iwasawa projection $k_1: G=KAN\to K$ in $e$ is the projection $\gg=\kk+\aa+\nn\to \kk$ along  $\aa+\nn$. The derivative of $\varphi$ in $e$ is the restriction of the composition of this projection with the canonical projection $\kk\to \kk/\mm$. Note that
$$E_{-\alpha}= E_\alpha+E_{-\alpha} -E_\alpha = E_\alpha^K- E_\alpha\in \kk +\nn, $$
so that $d\varphi(e)(E_{-\alpha})= E_\alpha^K+\mm$. As the $E_\alpha^K$ span $\mm^\perp\subseteq \kk$ whereas the $E_{-\alpha}$ span $\bar\nn$ this gives a complete description of $d\varphi(e)$. Note that $\varphi^*\eta_\alpha^K(e) =\eta_\alpha^K\circ d\varphi(e)$ and
$$\varphi^*\eta_\alpha^K(e)(E_{-\alpha'}) =\eta_\alpha^K(e)(E_{\alpha'}+E_{-\alpha'})
=\eta_\alpha^K(e)(E_{\alpha'}^K)
=\delta_{\alpha\alpha'}.
$$
This shows that $\varphi^*\eta_\alpha^K(e)=\eta_{-\alpha}(e)$ and hence by Remark~\ref{rem:volume forms X and B}(iii)
$$(\varphi^*\,  d\mathrm{vol}_{K/M})(e)
= 2^{\frac{1}{2}\dim\nn} \,  d\mathrm{vol}_{\bar N}(e).$$
Now the defining equation of $\tilde \psi$ gives
$$\tilde\psi(e)= 2^{\frac{1}{2}\dim\nn}, $$
so that we finally obtain 
$$c_{\bar N}=\frac{2^{\frac{1}{2}\dim\nn}}{\mathrm{vol}(K/M)},$$
which agrees with the normalization obtained from \cite[Prop.~2.5.7]{gangolli-varadarajan}. 
    \end{remark}

Finally we can calculate also $c_1$.

\begin{remark}
Combining Remark~\ref{rem:c_G} with Remark~\ref{rem:cN} yields
$$c_G=c_{\bar N} c_M c_A c_N = c_N^2 c_A c_M 
= 2^{\dim\nn}\Big(\frac{c_K}{c_M}\Big)^2 c_M c_A
= 2^{\dim\nn}\frac{c_A{c_K}^2}{c_M}.$$
On the other hand, let $\iota_{\bar N}:\bar N\to G$ be the inclusion. Since 
$$dp_B(e):\bar \nn \to \gg/(\mm+\aa+\nn)\cong \kk/\mm,\,\, E_{-\alpha}\mapsto E_{-\alpha}+\mm+\aa+\nn\cong E_\alpha+E_{-\alpha}+\mm,$$
we get 
$$(p_B\circ \iota_{\bar N})^*db 
= \frac{c_K}{c_M }(p_B\circ \iota_{\bar N})^*d\mathrm{vol}_{B} = 2^{\frac{1}{2}\dim\nn}e^{-\rho H(\bar n)}\frac{c_K}{c_M }\, d\mathrm{vol}_{\bar N}
.$$
Comparing this to 
Lemma~\ref{lemmadb} and its proof shows now that 

$$c_1=  2^{\frac{1}{2}\dim\nn}\frac{c_K}{c_M} =\frac{2^{\frac{1}{2}\dim\nn}}{\mathrm{vol}(K/M)}=\frac{2^{\frac{1}{2}\dim(K/M)}}{\mathrm{vol}(K/M)} .  $$
Recalling the formula for $c_0$ in Remark~\ref{rem:c0}, the formula for $c_N=c_{\bar N}$ in Remark~\ref{rem:cN}, the product formula \eqref{eq: c_G} and the formula for $c_{K/M}$ from Remark~\ref{rem:volGK}  we get
$$c_0 
= 2^{\frac{1}{2}\dim(\nn)}\,  \frac{c_G}{c_K} 
= 2^{\frac{1}{2}\dim(\nn)}\, 2^{\dim(\nn)}\Big(\frac{c_K}{c_M}\Big)^2\,  \frac{c_M c_A}{c_K} 
= 2^{\frac{3}{2}\dim(\nn)}\, \frac{c_A}{\mathrm{vol}(K/M)} .$$
The normalization for $da$ given in \cite[\S~II.3.1]{helgason:2008} is such that 
$$c_A=(2\pi)^{-\frac{1}{2}\dim\aa}=(2\pi)^{-\frac{r}{2}},$$
so that
$$c_0 = (2\pi)^{-\frac{1}{2}\dim \aa}2^{\frac{3}{2}\dim(\nn)}\, \frac{1}{\mathrm{vol}(K/M)} 
=\pi^{-\frac{1}{2}\mathrm{rank}(G/K)}\frac{2^{-\frac{1}{2}\mathrm{rank}(G/K)+\frac{3}{2}\dim(\nn)}}{\mathrm{vol}(K/M)}.$$
The constant $c_0c_1^{-1}$ showing up in Theorem~\ref{ThQG} is given by 
$$c_0c_1^{-1} = c_A 2^{\dim\nn}
= (2\pi)^{-\frac{\mathrm{rank}(G/K)}{2}}2^{\dim(K/M)}.$$
\end{remark}

\begin{remark}
    We apply the integral formula in \cite[Thm.~I.5.17]{helgason:1984} to rewrite the integral 
    $$\int_{\mathfrak s} e^{-|H|^2}\, d\,\mathrm{vol}_\mathfrak{s}(H)= \pi^{\frac{1}{2}\dim\mathfrak{s}},$$ where the norm on $\mathfrak s$ comes from the inner product $B_\sigma$ which on $\mathfrak s$ agrees with the Killing form. Note here that
    $$\dim\mathfrak{s} = \dim X= \dim \mathfrak n +\dim \mathfrak a.$$
Here the product over $\Sigma^+$
is over distinct restricted roots.
For the orientations fixed above,
the map $(kM,H)\mapsto\Ad_k H$
on $K/M\times\aa_+$
has signed Jacobian
$(-1)^d\prod_{\alpha\in\Sigma^+}
\alpha(H)^{m_\alpha}$,
since
$[E_\alpha^K/\sqrt2,H]
=-\alpha(H)E_\alpha^\ss/\sqrt2$.
The measure formula uses its absolute value.
    The integral formula yields
    \begin{align*}
        \pi^{\frac{1}{2}\dim\mathfrak{s}}
    &= \int_{K/M} \int_{\mathfrak{a}_+} e^{-|H|^2}\prod_{\alpha\in \Sigma^+} |\alpha(H)|^{m_\alpha}  d\,\mathrm{vol}_\mathfrak{a}(H)\, d\,\mathrm{vol}_{K/M}(k)\\
    &=\mathrm{vol}(K/M) \int_{\mathfrak{a}_+} e^{-|H|^2}\prod_{\alpha\in \Sigma^+} |\alpha(H)|^{m_\alpha}  d\,\mathrm{vol}_\mathfrak{a}(H).
    \end{align*}
    Thus we have written $\mathrm{vol}(K/M)$ purely in terms of root data.
\end{remark}

\bigskip

\bigskip

\begingroup
\sloppy

{\noindent \bf Acknowledgements:} The research of ACF was financed by Portuguese Funds through FCT (Fundação para a Ciência e a Tecnologia, I.P.) within the Project UID/00013/2025
(https://doi.org/10.54499/UID/00013/2025).
JH was partially funded by the Deutsche Forschungsgemeinschaft (DFG, German Research Foundation) – Pro\-ject-ID 491392403 – TRR 358.
JM and JN were supported by FCT/Portugal and the Recovery and Resilience Plan (PRR) through projects UID/04459/2025 and UID/PRR/04459/2025.

\endgroup

\bigskip

{\noindent \bf AI disclosure statement:}
ChatGPT-6 Astra was used to assist with LaTeX and, after the manuscript had been written, to check mathematical correctness. It identified errors in Lemmas~\ref{lem:stat-phase-compact-supp} and~\ref{lem:stat-phase-compact-schwartz} and suggested an approach to their correction, which the authors developed and verified. The authors take full responsibility for the contents of the paper.

\begingroup
\sloppy

\printbibliography
\endgroup

\end{document}